\documentclass[reqno]{amsart}  
\usepackage{graphicx}   
\usepackage{amsmath,amssymb,mathrsfs}    
\usepackage[colorlinks]{hyperref}
\usepackage{enumitem}

\newtheorem{theorem}{Theorem}[section] 
 
\newtheorem{lemma}[theorem]{Lemma}  
   
\newtheorem{corollary}[theorem]{Corollary}

\newtheorem{fact}[theorem]{Fact} 
\newtheorem{conjecture}[theorem]{Conjecture} 
    
\newtheoremstyle{definition}% name
  {4pt}%      Space above 
  {4pt}%      Space below
  {\sl}%      Body font
  {}%         Indent amount (empty = no indent, \parindent = para indent)
  {\bfseries}% Thm head font 
  {.}%        Punctuation after thm head
  {.5em}%     Space after thm head: " " = normal interword space;
  {}%         Thm head spec (can be left empty, meaning `normal')

\theoremstyle{definition}
\newtheorem{definition}[theorem]{Definition}

\theoremstyle{remark}
\newtheorem{remark}[theorem]{Remark}

\newtheoremstyle{introthms}% name
  {3pt}%      Space above, empty = `usual value'
  {3pt}%      Space below
  {\itshape}% Body font
  {}%         Indent amount (empty = no indent, \parindent = para indent)
  {\bfseries}% Thm head font
  {.}%        Punctuation after thm head
  {.5em}%     Space after thm head: " " = normal interword space;
  {\thmnote{#3}}% Thm head spec 
\theoremstyle{introthms}
\let\eps=\varepsilon
\let\theta=\vartheta
\let\rho=\varrho
\let\phi=\varphi

\def\colond{\colon\,}
\let\polishlcross=\l
\def\l{\ifmmode\ell\else\polishlcross\fi}

\newcommand{\dotcup}{\ensuremath{\mathaccent\cdot\cup}}
\def\dcup{\dotcup}

\def\rmlabel{\upshape({\itshape \roman*\,})}
\def\RMlabel{\upshape(\Roman*)}
\def\alabel{\upshape({\itshape \alph*\,})}
\def\Alabel{\upshape({\itshape \Alph*\,})}

\def\Qlabel{\upshape({\itshape Q\arabic*\,})}
\def\Rlabel{\upshape({\itshape R\arabic*\,})}

\def\tand{\ \text{and}\ }
\def\qand{\quad\text{and}\quad}

\def\NN{\mathbb N}
\def\ZZ{\mathbb Z}

\def\RR{\mathbb R}
\def\PP{\mathbb P}

\def\cG{{\mathcal G}}
\def\tcG{\widetilde {\mathcal G}}

\def\cB{{\mathcal B}}

\def\cP{{\mathcal P}}

\def\cR{{\mathcal R}}
\def\cS{{\mathcal S}}

\def\cK{{\mathcal K}}

\def\symd{\mathbin{\triangle}}

\DeclareMathOperator{\Forb}{Forb}
\DeclareMathOperator{\Col}{Col}
\DeclareMathOperator{\col}{col}

\DeclareMathOperator{\ex}{ex}

\def\ph{\hat p}

\def\euler{\textrm{e}}

\def\epsRL{\eps_{\textrm{RL}}}
\def\epsSTAB{\eps}

\def\DISC{\mathrm{DISC}}
\def\CL{\mathrm{EMB}}
\def\HCL{\mathrm{HEMB}}

\newcommand{\nocontentsline}[3]{}
\newcommand{\tocless}[2]{\bgroup\let\addcontentsline=\nocontentsline#1{#2}\egroup}

\begin{document}

%%%%%%%%%%%%%%%%%%%%%%%%%%%%%%%%%%%%%%%%%%%%%%%%%%%%%%%%

\title{Extremal results in random graphs}

\author[Vojt\v ech R\"odl]{Vojt\v ech R\"odl}
\address{Department of Mathematics and Computer Science, Emory
  University, Atlanta, GA 30322, USA}
\email{rodl@mathcs.emory.edu}
\thanks{First author was supported by NSF grant DMS~0800070.}

\author[Mathias Schacht]{Mathias Schacht}
\address{Fachbereich Mathematik, Universit\"at Hamburg,
  Bundesstra\ss{}e~55, D-20146 Hamburg, Germany}
\email{schacht@math.uni-hamburg.de}
\thanks{Second author was supported through the Heisenberg-Programme of the
Deutsche Forschungsgemeinschaft (DFG Grant SCHA 1263/4-1).}

\dedicatory{Dedicated to the memory of Paul Erd\H os on the occasion
     of his $100$th birthday}

\begin{abstract}
According to Paul Erd\H os [\emph{Some notes on {T}ur\'an's mathematical work}, J. Approx. Theory \textbf{29} (1980), page 4]  it was Paul Tur\'an who ``created the area of extremal problems in graph theory''. 
However, without a doubt, Paul Erd\H os popularized \emph{extremal combinatorics}, 
by his many contributions to the field, his numerous questions and conjectures, and his influence on
discrete mathematicians in Hungary and all over the world. 
In fact, most of the early contributions in this field 
can be traced back to Paul Erd\H os, Paul Tur\'an, as well as their collaborators and students.
Paul Erd\H os also
established the \emph{probabilistic method} in discrete mathematics,
and in collaboration with Alfr\'ed R\'enyi,
he started the systematic study of \emph{random graphs}.
We shall survey recent developments
at the interface of extremal combinatorics and random graph theory. 
\end{abstract}
\maketitle

%%%%%%%%%%%%%%%%%%%%%%%%%%%%%%%%%%%%%%%%%%%%%%%%%%%%%%%

\tableofcontents

\section{Extremal Graph Theory}
\subsection{Introduction}
We first discuss a few classical results in extremal graph theory. Since by no means we can give a full 
account
here, we restrict ourselves to some well known results in the area and highlight some of the pivotal
questions. 
For a thorough introduction to the area we refer to the standard textbook of 
Bollob\'as~\cite{Bo78}. 

A large part of extremal graph theory concerns the study of 
graphs~$G$ which do not contain a given subgraph~$F$.
The first classical problem is to maximize the number of edges of such a graph~$G$
with~$n$ vertices. An instance of this question was addressed already in 1938 by Erd\H os. In~\cite{Er38}
he proved bounds for an extremal problem in combinatorial number theory, and in his proof he asserts a 
lemma that every $n$-vertex graph without a cycle of length four can have at most $cn^{3/2}$ edges 
(see Figure~\ref{fig:Er38} below).
\begin{figure}[h]
\centering
\includegraphics[width=0.95\textwidth]{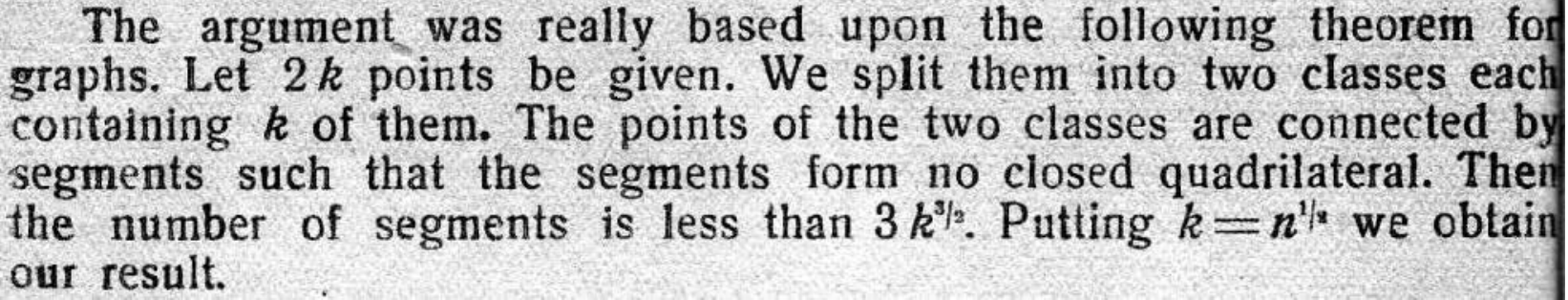}
\caption{Quote from~\cite[page 78]{Er38}}
\label{fig:Er38}
\end{figure}

Tur\'an initiated the systematic study of such questions, and in Section~\ref{sec:Turan} we give
a short account of Tur\'an's theorem~\cite{Tu41} in graph theory and some important results related 
to it. In fact, we will restrict ourselves
only to extremal questions in graph theory here. However, 
even within extremal graph theory we can only discuss a few selected results and are bound to neglect not only 
many important topics, but also many beautiful generalizations and improvements of those classical results. 
%For a thorough treatment the textbook of Bollob\'as~\cite{Bo78} is a good starting point.
Our certainly biased selection of results presented here
is guided by the recent generalizations, which were obtained for subgraphs of random graphs.
First we introduce the necessary notation.

\subsection{Notation}\label{sec:notation}
Below we recall some 
notation from graph theory, which will be used here. For notation not defined here we refer to
the standard text books~\cite{Bo98,BM08,Di10}. 

All graphs considered here are finite, simple and have no loops.
For a graph $G=(V,E)$ we denote by $V(G)=V$ and $E(G)=E$ its \emph{vertex set} and its 
\emph{edge set}, respectively. We denote by $e(G)=|E(G)|$ the number of edges of~$G$ and by $d(G)=e(G)/\binom{|V(G)|}{2}$ its \emph{edge density}. Moreover, for a subset $U\subseteq V$ let
$e_G(U)$ be the number of edges of $G$ contained in~$U$. By $\omega(G)$, $\alpha(G)$, and $\chi(G)$
we denote the standard graph parameters known as \emph{clique number}, \emph{independence number}, and 
\emph{chromatic number} of~$G$, respectively. We say  that a graph $G$ contains a \emph{copy} of a  graph $F$ if there is an injective map $\phi\colond V(F)\to V(G)$
such that $\{\phi(u),\phi(v)\}\in E(G)$, whenever $\{u,v\}\in E(F)$. If $G$ contains no such copy, then we say
$G$ is \emph{$F$-free}. Also, $G$ and $F$ are \emph{isomorphic} if there exists a bijection $\phi\colond G\to F$
such that $\{\phi(u),\phi(v)\}\in E(G)$ if, and only if $\{u,v\}\in E(F)$. In this case we often write $G=F$. A 
graph $H$ is a \emph{subgraph} of~$G=(V,E)$, if $V(H)\subseteq V$ and $E(H)\subseteq E$, which we denote by $H\subseteq G$.

The complete 
graph on $t$ vertices with $\binom{t}{2}$ edges is denoted by $K_t$, and a \emph{clique} is 
some complete graph. A graph $G$ 
is \emph{$t$-partite} or \emph{$t$-colorable}, if there is a partition of its 
vertex set into $t$ classes (some of them might be empty) such that every edge 
of~$G$ has its vertices in two different partition classes. We denote by $\Col_n(t)$ the set of all $t$-colorable 
graphs on~$n$ vertices, i.e., 
\[
	\Col_n(t)=\{H\subseteq K_n\colond \chi(H)\leq t\}\,.
\]
A $t$-partite graph 
$G=(V,E)$ with vertex classes $V_1\dcup \dots\dcup V_t=V$ is 
\emph{complete} if for every $1\leq i<j\leq t$ and every 
$u\in V_i$ and $v\in V_j$ we have  $\{u,v\}\in E$. We denote by $T_{n,t}$ the
complete $t$-partite graph on~$n$ vertices with the maximum number of edges.
It is easy to show that~$T_{n,t}$ is unique up to isomorphism and that it is the 
complete $t$-partite graph with every vertex class having cardinality either 
$\lfloor n/t\rfloor$ or $\lceil n/t\rceil$.

For a graph~$F$ with at least one edge and an integer~$n$,
we denote by $\Forb_n(F)$ the set of $F$-free subgraphs of~$K_n$, i.e.,
\[
	\Forb_n(F)=\{H\subseteq K_n\colond H\ \text{is $F$-free}\}\,,
\]  
and we recall the \emph{extremal function}~$\ex_n(F)$ defined by
\[
	\ex_n(F)=\max\{e(H)\colond H\in \Forb_n(F)\}\,.
\]
Note that the set $\Forb_n(F)$ is closed under taking subgraphs, i.e., if $H\in \Forb_n(F)$ and $H'\subseteq H$, then 
$H'\in \Forb_n(F)$.
In general such sets of graphs are called \emph{monotone}. In fact, any monotone property $\cP_n$
of subgraphs of $K_n$ can be expressed by a family of \emph{forbidden} subgraphs, and many results 
discussed below allow generalizations in this direction (and even more generally towards hereditary properties). 
However, we will concentrate on generalizations for subgraphs of random graphs and restrict the discussion 
to  a forbidden set of graphs consisting of only one graph.

\subsection{Tur\'an's Theorem and Related Results}
\label{sec:Turan}
Generalizing a result of 
Mantel~\cite{Ma07} for $F=K_3$, Tur\'an~\cite{Tu41} 
determined $\ex_n(F)$ when $F$ is a complete graph. 
\begin{theorem}[Tur\'an 1941]
\label{thm:Turan}
For all integers $t\geq 2$ and $n\geq 1$ we have 
\[
	\ex_n(K_{t+1})=e(T_{n,t})\,.
\]
Moreover, $T_{n,t}$ is, up to isomorphism, the unique $K_{t+1}$-free graph on $n$ vertices with $\ex_n(K_{t+1})$ edges.
\end{theorem}

Theorem~\ref{thm:Turan} determines the maximum number of edges of a $K_{t+1}$-free graph on $n$ vertices.
Moreover, it characterizes the \emph{extremal graphs}, i.e., those $K_{t+1}$-free graphs on $n$ vertices having the 
maximum number of edges. In fact, these are instances  of two very typical questions in extremal combinatorics. The questions
below are stated more generally and could be applied in other contexts like hypergraphs, multigraphs, subsets of the integers, etc. 
However, we shall mostly restrict ourselves to questions in graph theory here.
\begin{enumerate}[label=\Qlabel]
	\item\label{q:1} Given a monotone property of discrete structures, like the monotone set $\Forb_n(K_{t+1})$ of subgraphs
	of $K_n$, what maximum density can its members attain? 
	\item\label{q:2} What are the extremal discrete structures, e.g., like $T_{n,t}$ is the extremal subgraph of $K_n$ for $\Forb_n(K_{t+1})$?
\end{enumerate}
Theorem~\ref{thm:Turan} answers~\ref{q:1} and~\ref{q:2} in a precise way. In fact, it
not only determines the maximum density, as required for~\ref{q:1}, 
but actually gives a full description of the function $\ex_n(K_{t+1})$.
Often %such a precise answer is not known and
only the density question can be addressed.
% in an asymptotic way. 

To this end
for a given graph $F$ we recall the definition of the \emph{Tur\'an denisty~$\pi(F)$}, which is given 
by 
\[
	\pi(F)=\lim_{n\to\infty}\frac{\ex_n(F)}{\binom{n}{2}}\,.
\]
Note that the limit indeed exists since one can show that 
$\ex_n(F)/\binom{n}{2}$ is non-increasing in~$n$.
Erd\H os and Stone~\cite{ErSt46} determined $\pi(F)$ for every graph~$F$. 
\begin{theorem}[Erd\H os \& Stone, 1946]
	\label{thm:ES}
	For every graph~$F$ with at least one edge we have 
	\[
		\pi(F)=1-\frac{1}{\chi(F)-1}\,.
	\]
\end{theorem}
In particular, $\pi(F)=0$ for every bipartite graph~$F$ (see also~\cite{KST54} for stronger 
estimates for this problem).
On the other hand, for a graph $F$ of chromatic number at least three the lower bound in Theorem~\ref{thm:ES} is established by the Tur\'an graph $T_{n,\chi(F)-1}$.

Refining Theorem~\ref{thm:ES} by determining $\ex_n(F)$  for arbitrary~$F$ is a very hard problem 
(see, e.g.,~\cite{ErSi71,Si74a,Si74b} for some partial results in this direction). Consequently, 
a precise solution for question~\ref{q:2} is still unknown for most graphs~$F$. Owing to the \emph{stability theorem}, which was 
 independently obtained by Erd\H os~\cite{Er67} and Simonovits~\cite{Si68}, we however have an approximate answer for question~\ref{q:2}.
In fact, the stability theorem determines an approximate structure of the extremal, as well as
the \emph{almost extremal}, graphs up to $o(n^2)$ edges. 
\begin{theorem}[Erd\H os 1967, Simonovits 1968]
	\label{thm:stability}
	For every $\eps>0$ and every graph $F$ with $\chi(F)=t+1\geq 3$ there exist $\delta>0$ and~$n_0$ such that the following holds.
	If $H$ is an $F$-free graph  on $n\geq n_0$ vertices satisfying
	\[
		e(H)\geq \ex_n(F)-\delta n^2\,,
	\]  
	then there exists a copy $T$ of $T_{n,t}$ on $V(G)$ such that
	\[
		|E(H)\symd E(T)|\leq \eps n^2\,,
	\]
	where $\symd$ denote the symmetric difference of sets.
	
	In other words, $H$ can be obtained from the graph $T_{n,t}$ by adding and deleting up to at most 
		$\eps n^2$ edges.
		
	In particular,	$H$ can be made $t$-partite by removing at most $\eps n^2$ edges from it. 
\end{theorem}
Note that Theorem~\ref{thm:stability} holds trivially for bipartite graphs~$F$ as well, since in this case $\ex_n(F)=o(n^2)$, 
and $T_{n,1}$ corresponds to an independent set.

Next we state two more commonly asked questions in extremal combinatorics, which we shall discuss in the context of being $F$-free.
\begin{enumerate}[label=\Qlabel]
	\setcounter{enumi}{2}
	\item\label{q:3} How many discrete structures of given size have the monotone property? E.g., how large is the set $\Forb_n(F)$? 
	\item\label{q:4} Do the \emph{typical} (drawn uniform at random) 
	discrete structures with this property have some common features,?
		E.g., 
		are there any common features of almost all graphs in $\Forb_n(F)$?
\end{enumerate}

For $K_{t+1}$-free graphs both of these questions were addressed in the work of Erd\H os, Kleitman, and Rothschild~\cite{EKR76}
and Kolaitis, Pr\"omel, and Rothschild~\cite{KPR85,KPR87}. In particular, it was shown that \emph{almost all}
$K_{t+1}$-free graph on~$n$ vertices are %$t$-colorable. %We denote by $\Col_n(t)$ the set of 
$t$-colorable subgraphs of $K_n$.
\begin{theorem}[Kolaitis, Pr\"omel \& Rothschild, 1985]\label{thm:KPR}
For every integer $t\geq 2$ the limit 
$\lim_{n\to\infty}|\Forb_n(K_{t+1})|/|\Col_n(t)|$ exists and
\[
	\lim_{n\to\infty}\frac{|\Forb_n(K_{t+1})|}{|\Col_n(t)|}=1\,.
\]
\end{theorem}
Similarly to the extension of  Tur\'an's theorem in~\cite{Si74a}, Theorem~\ref{thm:KPR} was extended by Pr\"omel and 
Steger~\cite{PS92} from cliques $K_{t+1}$ to graphs \emph{containing a color-critical edge}, i.e., 
$(t+1)$-chromatic graphs~$F$ with the property that~$\chi(F-f)=t$ for some edge $f\in E(F)$ 
(see also~\cite{BBS09,BBS11} for more recent extensions of Theorem~\ref{thm:KPR}).

Regarding question~\ref{q:3}, for arbitrary graphs~$F$, the size of $\Forb_n(F)$ was studied by 
Erd\H os, Frankl, and R\"odl~\cite{EFR86}, and those authors arrived at the following estimate (see also~\cite{BBS04} for a 
more recent improvement).
\begin{theorem}[Erd\H os, Frankl \& R\"odl, 1986]\label{thm:EFR}
	For every $\eps>0$ and every graph~$F$ there exists $n_0$ such that for every $n\geq n_0$ we have
	\[
		|\Forb_n(F)|\leq 2^{\ex_n(F)+\eps n^2}\,.
	\]
\end{theorem}
Note that $|\Forb_n(F)|\geq 2^{\ex_n(F)}$ holds trivially, since every subgraph of an extremal graph on $n$ vertices 
is $F$-free. Therefore, Theorem~\ref{thm:EFR} implies for every graph~$F$ that
\[
	\lim_{n\to\infty}\frac{\log_2|\Forb_n(F)|}{\binom{n}{2}}=\pi(F)\,.
\]
The extremal results stated above
were motivated by Tur\'an's theorem, and the problems addressed by those results 
allow natural generalizations for subgraphs of random graphs. In the next section  we consider 
such extensions, where the complete graph $K_n$ (in the definition of $\ex_n(F)$ and $\Forb_n(F)$)
is replaced by a \emph{random graph} with vanishing edge density. We will discuss some 
further extremal results, including the \emph{removal lemma} and the \emph{clique density theorem} 
in Section~\ref{sec:apps}.

\section{Extremal Problems for Random Graphs}
Motivated by questions in Ramsey theory (also known as \emph{Folkman-type problems}),
in 1983, at the first \emph{Random Structures and Algorithms} conference in 
Pozna\'n,  Erd\H os and Ne\v set\v ril (see~\cite{Er83}) posed the following extremal problem:
Is it true that for every $\eps>0$ 
there exists a $K_4$-free graph~$G$ such that any subgraph $H\subseteq G$ containing at least 
$(1/2+\eps)e(G)$ edges must contain a triangle? In other words,  Erd\H os and 
Ne\v set\v ril asked whether for $F=K_3$ one may replace~$K_n$ in the Erd\H os--Stone theorem 
by a graph which contains no larger cliques than the triangle itself.
This question was answered positively by Frankl and R\"odl~\cite{FR86}
by a random construction. Those authors considered the \emph{binomial random graph}~$G(n,p)$ 
with vertex set $[n]=\{1,\dots,n\}$, in which the edges are chosen independently, each with, probability~$p$
(see, e.g.,~\cite{Bo01,JLR00} for standard textbooks on the topic). 
More precisely, it was shown that for $p=n^{-1/2+o(1)}$ 
a.a.s.\ one may remove $o(pn^2)$ edges from $G\in G(n,p)$  (one from every copy of $K_4$ in $G$) 
such that the remaining graph has the desired property. In particular, a.a.s.\ 
the largest triangle-free subgraph of $G(n,p)$ contains at most $(\pi(K_3)+o(1))p\binom{n}{2}$ edges
(see Theorem~\ref{thm:FR86} below).

It will be convenient to extend the definitions $\Col_n(t)$, $\Forb_n(F)$ and $\ex_n(F)$ from Section~\ref{sec:notation}
to a more general setting.
For a graph $G$ and an integer $t$ let
\[
	\Col_G(t)=\{H\subseteq G\colond \chi(H)\leq t\}
\]
be the set of $t$-colorable subgraphs of~$G$.
Similarly, for a graph  $F$ with at least one edge we denote by $\Forb_G(F)$ the set of all 
subgraphs of $G$ not containing a copy of~$F$, i.e., 
\[
		\Forb_G(F)=\{H\subseteq G\colond H\ \text{is $F$-free}\}\,,
\]
and we define the \emph{generalized extremal function}
$\ex_G(F)$ as the maximum number of edges 
of the elements of $\Forb_G(F)$, i.e.,
\[
	\ex_G(F)=\max\{e(H)\colond H\in\Forb_G(F)\}\,.
\]
%Note that with this notation we have the obvious identities
%\[
%	\ex_n(F)=\ex_{K_n}(F)
%	\qand
%	\Forb_n(F)=\Forb(K_{n},F)\,.
%\]
The following was proved by Frankl and R\"odl in~\cite{FR86}.
\begin{theorem}
	\label{thm:FR86}
	Let $\eps>0$ and $p\geq n^{-1/2+\xi}$ for some $\xi>0$. Then a.a.s.\ for $G\in G(n,p)$ we have 
	$\ex_G(K_3)\leq(\pi(K_3)+\eps)e(G)$. 
\end{theorem}
In view of Theorem~\ref{thm:FR86} several questions arise (see below). The systematic study of these questions
was initiated by the work of Kohayakawa and his collaborators in~\cite{HKL95,HKL96,KKS98,KLR97,Lu00}. In particular,
Kohayakawa, \L uczak, and R\"odl formulated conjectures in~\cite{KLR97}, which led to the subsequent work discussed here.
\begin{enumerate}[label=\Rlabel]
	\item \label{q':1} What is the smallest~$p$ such that Theorem~\ref{thm:FR86} holds?
	\item \label{q':2} For which~$p$ can Theorem~\ref{thm:FR86} be extended to other graphs $F$ instead of the triangle?
	\item \label{q':4} Are there stability versions for those results?
	\item \label{q':3} Is there a strengthening of Theorem~\ref{thm:FR86} which, similarly as 
	Mantel's theorem (Theorem~\ref{thm:Turan} for $t=2$), 
			establishes the equality between the maximum size of a bipartite subgraph and that of a triangle-free subgraph? More precisely, 
			what is the smallest $p$ such that a.a.s.\ $G\in G(n,p)$ has the following property: every $H\in\Forb_G(K_3)$ with 
			$e(H)=\ex_G(K_3)$ is bipartite?
	\item \label{q':5} What can be said about extensions of Theorems~\ref{thm:KPR} and~\ref{thm:EFR}, where instead of $K_{t+1}$-free subgraphs of $K_n$, 
		one studies $K_{t+1}$-free subgraphs of $G(n,p)$ for appropriate $p$?  Are almost all of those $t$-partite or ``close'' to being $t$-partite?
\end{enumerate}

We will address questions~\ref{q':1}--\ref{q':4} in the next section, Section~\ref{sec:probES}.
In Section~\ref{sec:EN} we address a generalization of the  question of Erd\H os--Ne\v set\v ril
that led to Theorem~\ref{thm:FR86}.
Results addressing question~\ref{q':3} will be discussed in Section~\ref{sec:probMantel} and then we turn to
question~\ref{q':5} in Section~\ref{sec:probKPR}.

\subsection{Threshold for the Erd\H os--Stone Theorem}
\label{sec:probES}
A common theme in the theory of random graphs is the \emph{threshold phenomenon}. For example, it was already observed by Erd\H os and Whitney (unpublished) and Erd\H os
and R\'enyi~\cite{ERe59} that within a ``small range'' of~$p$ (around $\ln n/n$) 
the random graph $G(n,p)$ quickly changes  
its behavior from being a.a.s.\ disconnected %(for $p<\ln n/n$) 
to being a.a.s.\ connected. %(for $p> \ln n/n$).
In other words, $\ph=\ln n/n$ is the \emph{threshold} for $G(n,p)$ being connected. 
In more generality, 
for a \emph{graph property~$\cP$}, i.e, a set of graphs closed under isomorphism, we say 
$0\leq \ph=\ph(n)\leq 1$ is a \emph{threshold function} for $\cP$, if 
\begin{equation}\label{eq:threshold}
	\lim_{n\to\infty}\PP(G(n,p)\in\cP)
		=\begin{cases}
			0, & \text{if}\ p\ll\ph\,,\\
			1, & \text{if}\ p\gg\ph\,.
	\end{cases}
\end{equation}
We refer to the two statements involved in this definition as the \emph{$0$-statement} and the \emph{$1$-statement} of the threshold. 
It is well known (see, e.g.,~\cite{BT87}) that every \emph{monotone property~$\cP$} has a threshold.

In Theorem~\ref{thm:FR86} the following property is studied for $F=K_3$. For given $\eps>0$, a graph $F$ with at least one edge,
and an integer $n$, consider
\[
	\cG_n(F,\eps)=\{G=(V,E)\colond V=[n] \tand \ex_G(F)\leq (\pi(F)+\eps)e(G)\}\,.
\]
We note that $\cG_n(F,\eps)$ is not monotone. Consider, for example, the case when $F=K_3$, and 
let $G\subseteq G'$ be graphs with vertex set $[n]$, where
$G$ consists of a clique on $n^{1/3}$ vertices all other vertices isolated and $G'$ consists of the union of $G$ and
a perfect matching.

Since $\cG_n(F,\eps)$ is not monotone, the threshold is not guaranteed to exist by the 
aforementioned result from~\cite{BT87}. 
On the other hand, $\cG_n(F,\eps)$ is 
``probabilistically monotone'' (see, e.g.,~\cite[Proposition~8.6]{JLR00}),
and from this it follows that indeed it has a threshold for all non-trivial $F$ and $\eps>0$. 
In view of this, questions~\ref{q':1} and~\ref{q':2} 
ask to determine the threshold for $\cG_n(K_3,\eps)$ and, more generally, for $\cG_n(F,\eps)$ 
for general~$F$.

Concerning the threshold for $\cG_n(K_3,\eps)$, it follows from Theorem~\ref{thm:FR86}
that for every $\eps>0$ this threshold is at most $\ph< n^{-1/2+\eps}$. However, a more careful analysis of  
the proof presented in~\cite{FR86} yields  $O(n^{-1/2})$ as an upper bound 
(see, e.g.,~\cite[Section~8.2]{JLR00}). For the lower bound on the threshold, we note that 
the expected number of triangles in $G(n,p)$ for $p=o(n^{-1/2})$ is 
$o(pn^{2})$. Hence by removing from $G(n,p)$
one edge from every triangle, we expect to be left with a triangle-free subgraph of $G(n,p)$ containing $1-o(1)$ 
proportion of the edges of $G(n,p)$. In fact, this argument can be made precise, and it follows that $\ph=n^{-1/2}$ is a threshold 
for $G(n,p)\in \cG_n(K_3,\eps)$ for every~$\eps>0$, which answers question~\ref{q':1}. (We remark that, in particular, the threshold function $\ph$ 
is independent of~$\eps$).

Regarding question~\ref{q':2}, we note that the lower bound for the threshold discussed above can 
be extended to arbitrary graphs and leads to the definition 
of the \emph{$2$-density $m_2(F)$} of a graph~$F$ with at least one edge given by 
\begin{equation}
	m_2(F)=\max\{d_2(F')\colond F'\subseteq F\ \text{with}\ e(F')\geq 1\}\,,\label{eq:m2}
\end{equation}
where
\[
	d_2(F')=\begin{cases}
		\frac{e(F')-1}{|V(F')|-2}, &\text{if}\ |V(F')|>2\,,\\
		1/2, &\text{if}\ F'=K_2\,.
	\end{cases}
\]
We say a graph \emph{$F$ is $2$-balanced} if $d_2(F)=m_2(F)$ and it is \emph{strictly $2$-balanced} if 
$d_2(F')<d_2(F)=m_2(F)$ for all subgraphs $F'\subsetneq F$ with at least one edge.
 
It follows from the definition of the $2$-density that $p=\Omega(n^{-1/m_2(F)})$ if, and only if
the expected numbers of copies of $F$ or any of its subgraphs in $G(n,p)$  
is at least of order $\Omega(pn^2)$ -- the order of the expected number of edges in $G(n,p)$. Similarly as above, one can 
deduce that for every $\eps>0$ and every graph $F$ with at least one edge, 
$n^{-1/m_2(F)}$ is a lower bound for the threshold for $\cG_n(F,\eps)$. Moreover, Kohayakawa, \L uczak, and R\"odl~\cite[Conjecture~1(\textit{i\,})]{KLR97}
conjectured that this heuristic gives the right bound, and that a matching upper bound for the threshold can be proved.
Until recently this conjecture was only proved for cliques of size at most six~\cite{KLR97,GSS04,Ge05} and for 
cycles~\cite{HKL95,HKL96}. In 2009 the conjecture was confirmed independently by Conlon and Gowers~\cite{CG} for 
strictly $2$-balanced graphs $F$ and by Schacht~\cite{Sch} for all graphs~$F$. This work yields the following probabilistic version 
of the Erd\H os--Stone theorem for the random graph~$G(n,p)$.
\begin{theorem}\label{thm:probES}
	For every graph~$F$ with $\delta(F)\geq 2$ and every $\eps>0$ 
	the function $\ph=n^{-1/m_2(F)}$ is 
	a threshold for $\cG_n(F,\eps)$.	
\end{theorem}

Next we discuss research addressing question~\ref{q':4}.
Recall that every graph~$G$ contains a $t$-partite subgraph with at least $(1-1/t)e(G)$ edges, which 
is clearly $F$-free for every~$F$ with chromatic number~$t+1$. On the other hand, the $1$-statement (see~\eqref{eq:threshold})
of Theorem~\ref{thm:probES} implies that a.a.s.\ the $F$-free subgraph of $G\in G(n,p)$ with the maximum number of edges has 
at most $(1-1/t+o(1))e(G)$ edges. The question that arises is whether those two subgraphs of $G(n,p)$, the maximum $t$-partite subgraph and the 
maximum $F$-free subgraph, have similar structure. It was conjectured by Kohayakawa, \L uczak, and 
R\"odl~\cite[Conjecture 1(\textit{ii\,})]{KLR97} that such a statement is true as long as~$p$ 
is of the order of magnitude given in the $1$-statement of the threshold in Theorem~\ref{thm:probES}.  
Conlon and Gowers~\cite{CG} verified this conjecture 
for strictly $2$-balanced graphs~$F$, and Samotij~\cite{Sa} adapted and simplified the 
approach of Schacht~\cite{Sch} to obtain such a result for all graphs~$F$. This led to the following probabilistic version of the Erd\H os--Simonovits 
stability theorem.
\begin{theorem}\label{thm:prob-stability}
		For every $\eps>0$ and every graph~$F$ with $\chi(F)=t+1\geq 3$ there exist 
		constants $C$ and $\delta>0$ such that for $p>Cn^{-1/m_2(F)}$ the following 
		holds a.a.s.\ for $G\in G(n,p)$.
		If $H$ is an $F$-free subgraph of~$G$ satisfying
		\[
			e(H)\geq \ex_G(F)-\delta p n^2\,,
		\]
		then $H$ can be made $t$-partite by removing at most $\eps pn^2$ edges from it. 
\end{theorem}

We recall that Theorems~\ref{thm:probES} and~\ref{thm:prob-stability} were conjectured (together with 
Conjecture~\ref{conj:KLR} stated in Section~\ref{sec:spRL}) in~\cite{KLR97}. These conjectures
played a central r\^ole in the area. In particular, partial results towards these conjectures were 
made by the authors of the conjecture and their collaborators~\cite{GKRS07,KR03,KR03s,KRS04}, by Gerke and 
Steger and their collaborators~\cite{Ge05,GMS07,GPSST07,GSS04,GS07} (see also the survey~\cite{GS05}), 
and by Szab\'o and Vu~\cite{SV03}.

\subsection{General Erd\H os--Ne\v set\v ril Problem}
	\label{sec:EN}
Before we continue with the discussion of extremal results for sparse random graphs, we generalize the problem of Erd\H os and Ne\v set\v ril.
Based on Theorem~\ref{thm:probES}, one now can prove the following generalization of the Erd\H os--Ne\v set\v ril problem, which extends
the results of~\cite{FR86} from forbidding triangles to forbidding cliques of arbitrary fixed size.
\begin{corollary}\label{cor:EN}
	For every integer $k\geq 3$ and $\eps\in (0,1-\pi(K_k))$
	the following holds:
	\begin{enumerate}[label=\rmlabel]
		\item \label{it:EN1} there exists a $K_{k+1}$-free graph~$G$ such that $\ex_G(K_{k})\leq (\pi(K_k)+\eps)e(G)$;
		\item \label{it:EN2} for every fixed $d>0$ there exists an $n_0$ such that there is no graph $G$ on $n\geq n_0$ vertices 
			with $e(G)=d\binom{n}{2}$ having the properties from part~\ref{it:EN1}.
	\end{enumerate}
\end{corollary}

While the first statement of Corollary~\ref{cor:EN} asserts the existence of a $K_{k+1}$-free graph with the property that every $(\pi(K_k)+\eps)$
proportion of its edges contains a~$K_k$, the second statement asserts that such a graph must have vanishing density.

In the proof of part~\ref{it:EN1} we consider $G(n,p)$ for $p=Cn^{-1/m_2(K_{k})}$. Owing to Theorem~\ref{thm:probES} 
we know that a.a.s.\ $G\in G(n,p)$ satisfies $\ex_G(K_{k})\leq (\pi(K_k)+o(1))e(G)$. On the other hand, since 
$n^{-1/m_2(K_k)} \ll  n^{-1/m_2(K_{k+1})}$ for this choice of $p$, the number of copies of $K_{k+1}$ in $G$ will 
be of order $o(pn^2)$. Consequently, we may remove $o(pn^2)$ edges from $G$ and the resulting graph is $K_{k+1}$-free 
and satisfies the properties of part~\ref{it:EN1} of Corollary~\ref{cor:EN}. In fact, one may check that the same proof 
works for all values of $p$ with $Cn^{-1/m_2(K_k)}\leq p\leq cn^{-1/m_2(K_{k+1})}$ for appropriate constants $C$ and $c>0$.
We give the details of this proof after the following remark.

\begin{remark}\textsl
One can show that statement~\ref{it:EN2} is best possible. Indeed, given 
$d=d(n)=o(1)$, let $(G_m)_{m\in \NN}$ be a sequence of $m$ vertex graphs 
with the properties of part~\ref{it:EN1} and with density $\rho=\rho(m)\ll d(m)$. 
Since $d=o(1)$, we can find  infinitely many values for which $d(n)\sim\rho(m)$.
For such an $m$ we ``blow-up'' $G_m$ by replacing each vertex by an independent set of size
$n/m$ and every edge by a complete bipartite graph with vertex classes of size $n/m$.
The resulting graph $G$ has $n$ vertices, density approximately $d(n)$, and it ``inherits'' 
the properties of $G_m$ with respect to statement~\ref{it:EN1}. 
\end{remark}

%Moreover, considering blow-ups of such graphs shows that for any density~$d$ with 
%$Cn^{-1/m_2(K_k)}\leq d=d(n)=o(1)$ there exists a graph having the properties of part~\ref{it:EN1}. This shows that 
%part~\ref{it:EN2} is best possible. 

Finally, we remark that in Section~\ref{sec:CGW} we will generalize part~\ref{it:EN2} 
and show that no relatively dense subgraph of $G(n,p)$ for $p\gg n^{-1/m_2(K_{k+1})}$ satisfies the 
properties of part~\ref{it:EN2} (see Theorem~\ref{thm:probEN}).

\begin{proof}[Proof of Corollary~\ref{cor:EN} part~\ref{it:EN1}]
Part~\ref{it:EN1} follows directly from Theorem~\ref{thm:probES} combined with an alteration argument (similar to the 
one carried out by Erd\H os in~\cite{Er59}, see also~\cite[Section~3]{AlSP08}). 

Let $\eps>0$ and $k\geq 3$ 
be given. Applying Theorem~\ref{thm:probES} for $\eps/2$ and $F=K_k$ implies that there exists a constant 
$C>0$ such that for $p=p(n)=Cn^{-1/m_2(K_k)}$ a.a.s.\ for $G\in G(n,p)$ we have 
\begin{equation}
	\label{eq:EN1a}
	\ex_G(K_k)\leq \left(\pi(K_k)+\frac{\eps}{2}\right)e(G)\,.
\end{equation}
Since $k\geq 3$, we have 
\[
	m_2(K_k)=\frac{\binom{k}{2}-1}{k-2}=\frac{k+1}{2}\geq 2\,,
\]
and thus also
\[
	p=Cn^{-\frac{2}{k+1}}\geq \frac{C}{\sqrt{n}}\,.
\] 
It follows that $pn^2\geq Cn^{3/2}$. Chebyshev's inequality easily yields that a.a.s.\ 
\begin{equation}
	\label{eq:EN1b}
	e(G)\geq \frac{1}{2}p\binom{n}{2}\,.
\end{equation}
Finally, we note that the expected number of copies of $K_{k+1}$ in $G$ is at most
\[
	p^{\binom{k+1}{2}}n^{k+1}
	%=
	%(p^{k}n)(p^{\binom{k}{2}}n^k)
	=C^{\binom{k+1}{2}} n\leq \frac{\eps}{4}p\binom{n}{2}\,,
\]
for sufficiently large~$n$. Hence it follows from Markov's inequality that, 
with probability at least $1/2$, the graph $G$ contains at most 
$(\eps/2)p\binom{n}{2}$ copies of $K_{k+1}$. Consequently, for sufficiently large~$n$ 
there exists a graph~$G$ containing
at most  $(\eps/2)p\binom{n}{2}$ copies of $K_{k+1}$ and for which~\eqref{eq:EN1a} 
and~\eqref{eq:EN1b} also hold. Let $G'$ be the graph obtained from $G$ by removing one edge from every copy
of~$K_{k+1}$ in~$G$. Obviously, the graph~$G'$ is $K_{k+1}$-free,
\[
	e(G')\geq e(G)-\frac{\eps}{4}p\binom{n}{2}\,,
\]
and owing to
\begin{align*}
	(\pi(K_k)+\eps)e(G') 
	&>
	(\pi(K_k)+\eps)e(G)-\frac{\eps}{4}p\binom{n}{2}\\
	&=
	\left(\pi(K_k)+\frac{\eps}{2}\right)e(G)+\frac{\eps}{2}e(G)-\frac{\eps}{4}p\binom{n}{2}\\
	&\overset{\eqref{eq:EN1b}}{\geq} 
	\left(\pi(K_k)+\frac{\eps}{2}\right)e(G)\,,
\end{align*}
it follows from~\eqref{eq:EN1a} that $\ex_{G'}(K_k)\leq (\pi(K_k)+\eps)e(G')$, which concludes the proof 
of assertion~\ref{it:EN1} in Corollary~\ref{cor:EN}.
\end{proof}

Next we prove assertion~\ref{it:EN2}. The proof follows the main ideas of~\cite[Theorem~4]{FR86}.
\begin{definition}\label{def:tcut}
	For a graph $G=(V,E)$ we call a partition $V_1\dcup\dots\dcup V_t=V$ a \emph{$t$-cut}. 
	We denote by $E_G(V_1,\dots,V_t)$ the \emph{edges of the $t$-cut}, i.e.,  
	those edges of~$G$ with its vertices in two different sets of the partition and we denote by
	$e_G(V_1,\dots,V_t)=	|E_G(V_1,\dots,V_t)|$ the \emph{size of the $t$-cut}.
	Moreover, we say a $t$-cut
	is \emph{balanced}, if $|V_1|\leq \dots \leq|V_t|\leq |V_1|+1$. 
\end{definition} 
A simple averaging argument shows that there always exists a balanced $t$-cut of~$G$ of size at least $(1-1/t)e(G)$.
The following lemma, which implies assertion~\ref{it:EN2} of Corollary~\ref{cor:EN}, shows that if 
on the other hand all balanced $t$-cuts have size at most $(1-1/t+o(1))e(G)$, then~$G$ contains cliques 
of arbitrary size. 
\begin{lemma}\label{lem:EN2}
	For all integers $s$, $t\geq 2$ and every $d>0$ there exist $\eps>0$ and $n_0$ such that the following holds.
	Let $G=(V,E)$ be a graph on $|V|=n\geq n_0$ vertices, with $|E|=d\binom{n}{2}$ edges, and with the property that 
	every balanced $t$-cut has size at most $(1-1/t+\eps)d\binom{n}{2}$. Then $G$ contains a copy of~$K_s$.
\end{lemma}

Before we prove Lemma~\ref{lem:EN2}, we deduce assertion~\ref{it:EN2} of Corollary~\ref{cor:EN} from it.
\begin{proof}[Proof of Corollary~\ref{cor:EN} part~\ref{it:EN2}] Suppose that part~\ref{it:EN2} of Corollary~\ref{cor:EN}
	fails to be true. We assume that
	there is a $K_{k+1}$-free graph $G$ on $n$ vertices with $\ex_n(K_k)\leq (\pi(K_k)+\eps)e(G)$ and with
	$d\binom{n}{2}$ edges for some constant $d>0$.
	We apply Lemma~\ref{lem:EN2} with $s=k+1$ and $t=k-1$. Since the edges of every $(k-1)$-cut span no copy of~$K_k$
	the assumption of  Corollary~\ref{cor:EN} part~\ref{it:EN2} guarantees that the size of every $(k-1)$-cut in~$G$
	is bounded from above by 
	\[
		(\pi(K_k)+\eps)e(G)=\left(1-\frac{1}{k-1}+\eps\right)d\binom{n}{2}\,,
	\]
	and it follows from Lemma~\ref{lem:EN2} that $G$ contains a $K_{k+1}$, which contradicts the assumption on~$G$.
\end{proof}

 The proof of Lemma~\ref{lem:EN2} draws on some ideas from the theory of quasi-random graphs~\cite{CGW89}. In particular, it is based on 
 the following well known fact (see, e.g.,~\cite[Theorem~2]{Ro86}).

\begin{lemma}\label{lem:Ro86}
	For all integers $s$, $t\geq 2$ and every $d>0$ there exist $\delta>0$ and~$n_0$ such that the following holds.
	Let $G=(V,E)$ be a graph on $|V|=n\geq n_0$ vertices such that $e_G(U)=(d\pm\delta)\binom{\lfloor n/t\rfloor}{2}$
	for every $U\subseteq V$ with $|U|=\lfloor n/t\rfloor$. Then $G$ contains a copy of~$K_s$.
\end{lemma}

\begin{proof}[Proof of Lemma~\ref{lem:EN2}]
	Let integers $s$ and $t\geq 2$ be fixed. Suppose for a contradiction that the lemma fails to be true 
	with this choice of~$s$ and~$t$. This means that there is a density $d>0$ for which the statement 
	fails, so
	we fix ``the largest such $d$''. More precisely, let $d>0$ be chosen in such a way that the statement fails for
	$s$, $t$, and $d$, but it holds for $s$, $t$ and any $d'> d$ provided $\eps'>0$ is sufficiently small and $n$ is sufficiently large. 
	We remark 
	that such a choice is possible, since for fixed $s$ and $t$ the validity of the statement for~$d$ implies it for every $d'\geq d$.
	
	Our choice of $\delta'$ will be given by Lemma~\ref{lem:Ro86}. First let $\delta>0$ be the constant guaranteed by Lemma~\ref{lem:Ro86}
	for the already fixed $s$, $t$, and $d$ and set 
	\begin{equation}\label{eq:EN2a}
		\delta'=\frac{\delta}{2(t-1)} \qand
		\eps= \min\left\{\frac{\delta}{4t^2}\,,\frac{\eps'(d+\delta')}{4t^2}\right\}\,,
	\end{equation}
	where $\eps'>0$ is given by Lemma~\ref{lem:EN2} applied with $d'\geq d+\delta'$ (which holds by our assumption).
	Finally, let $n_0$ be sufficiently large (for example, so that we can appeal to Lemma~\ref{lem:Ro86} with $s$, $t$, $d$, and $\delta$
	and to the validity of Lemma~\ref{lem:EN2} for $d'\geq d+\delta'$ and $\eps'>0$).
	Let $G=(V,E)$ with $|V|=n\geq n_0$ be a counterexample for those choices.	Without loss of generality we assume 
	that~$t^2$ divides~$n$.
	
	Since $G$ contains no copy of $K_s$, Lemma~\ref{lem:Ro86} implies that there exists a set $V_1$ of size $n/t$
	such that either $e_G(V_1)<(d-\delta)\binom{n/t}{2}$ or $e_G(V_1)>(d+\delta)\binom{n/t}{2}$.
	Fix some balanced $t$-cut $V_1\dcup\dots\dcup V_t=V$ which contains~$V_1$. We will infer that $G$ induces a denser graph on one 
	of the sets of the partition. This is obvious if $e_G(V_1)>(d+\delta)\binom{n/t}{2}$. However, if 
	$e_G(V_1)<(d-\delta)\binom{n/t}{2}$, then we will show that there also is a partition class that induces a denser graph.
	In fact, using the assumption on $G$  for the sizes of the balanced $t$-cuts, an averaging argument shows
	that there exists some $i=2,\dots,t$ such that 
	\begin{align*}
		e_G(V_i)
		&\geq
		\frac{e(G)-e_G(V_1,\dots,V_t)-e_G(V_1)}{t-1}\\
		&=
		\frac{d\binom{n}{2}-\left(1-\frac{1}{t}+\eps\right)d\binom{n}{2}-(d-\delta)\binom{n/t}{2}}{t-1}\\
		&\geq
		\frac{(1/t-\eps)d\binom{n}{2}-d\binom{n/t}{2}+\delta\binom{n/t}{2}}{t-1}\\
		&\geq
		\frac{(t-1-2\eps t^2)d\binom{n/t}{2}+\delta\binom{n/t}{2}}{t-1}\\
		&\overset{\eqref{eq:EN2a}}{\geq}
		(d+\delta')\binom{n/t}{2}\,.
	\end{align*}
	Summarizing, we can fix some $i\in[t]$ such that for $W=V_i$ we have $e_G(W)=d'\binom{n/t}{2}$ for some $d'\geq  d+\delta'$.
	
	Since $G$ (and hence also the induced subgraph $G[W]$) contains no copy of $K_s$, 
	by our assumptions $G[W]$
	fails to satisfy the assumptions of Lemma~\ref{lem:EN2}. Consequently, there exists a balanced $t$-cut 
	$W_{1}\dcup\dots\dcup W_{t}=W$ with 
	\begin{equation}\label{eq:EN2W}
		e_{G[W]}(W_{1},\dots,W_{t})>\left(1-\frac{1}{t}+\eps'\right)d'\binom{n/t}{2}=\left(1-\frac{1}{t}+\eps'\right)e_G(W)\,.
	\end{equation}
	We will extend this balanced $t$-cut of $G[W]$ to a balanced $t$-cut of $G$ with size bigger than
	\begin{equation}\label{eq:EN2goal}
		\left(1-\frac{1}{t}+\eps\right)d\binom{n}{2}\,,
	\end{equation}
	which will then contradict the assumptions on~$G$.
	
	For that we consider a random balanced $t$-cut $U_1\dcup\dots\dcup U_t$ of $U=V\setminus W$. A standard application of Chernoff's 
	inequality for the hypergeometric distribution (see, e.g.,~\cite[Theorem~2.10]{JLR00}) 
	shows that with probability close to one, we have
	\begin{equation}\label{eq:EN2WU}
		e_G(W_i,U_j)=\left(\frac{1}{t}\pm o(1)\right)e_G(W_i,U)\quad\text{for all $i$, $j\in[t]$}
	\end{equation}
	and
	\begin{equation}\label{eq:EN2U}
		e_G(U_1,\dots,U_t)=\left(1-\frac{1}{t}\pm o(1)\right)e_G(U)\,.
	\end{equation}
	Let such a $t$-cut be fixed.
	Since both $t$-cuts were balanced, the $t$-cut $V'_1\dcup\dots\dcup V'_t$ of $V$ given by $V'_i=W_i\dcup U_i$ is 
	also balanced. We estimate the size of this cut as follows:
	\begin{equation}\label{eq:EN2V'}
		e_G(V'_1,\dots,V'_t)=e_G(W_1,\dots,W_t)+\sum_{i=1}^t\sum_{j\neq i}e_G(W_i,U_j)+e_G(U_1,\dots,U_t)\,.
	\end{equation}
	By~\eqref{eq:EN2WU}, we have 
	\begin{align*}
		\sum_{i=1}^t\sum_{j\neq i}e_G(W_i,U_j)
		&=\sum_{i=1}^t\sum_{j\neq i}\left(\frac{1}{t}\pm o(1)\right)e_G(W_i,U) \\
		&=\left(\frac{t-1}{t}\pm o(1)\right)\sum_{i=1}^te_G(W_i,U)
		=\left(1-\frac{1}{t}\pm o(1)\right) e_G(W,U)
	\end{align*}
	and combined with~\eqref{eq:EN2W} and~\eqref{eq:EN2U}, from~\eqref{eq:EN2V'} we get
	\[
		e_G(V'_1,\dots,V'_t)\geq \left(1-\frac{1}{t}-o(1)\right)e(G)+\eps'e_G(W)\,.
	\]
	Hence, we obtain~\eqref{eq:EN2goal} from
	\[
		\eps'e_G(W)=\eps'd'\binom{n/t}{2}\geq \frac{\eps'd'}{2t^2}\binom{n}{2}\overset{\eqref{eq:EN2a}}{\geq} 2\eps\binom{n}{2}\,.
	\]
\end{proof}
	
\subsection{Turan's Theorem for Random Graphs}
\label{sec:probMantel}
Tur\'an's theorem not only determines the extremal function~$\ex_n(K_{t+1})$ precisely,
but also asserts that the complete balanced $t$-partite graph on~$n$ vertices is the unique extremal 
graph. The extremal results for $G(n,p)$ discussed in Section~\ref{sec:probES} do not fully address 
this question (see also~\ref{q':3}). For example, Theorem~\ref{thm:probES} applied for $F=K_{t+1}$
gives no information about the structure of extremal $K_{t+1}$-free subgraphs of $G(n,p)$.
In this section, we discuss results motivated by this question.

For an integer $t\geq 2$ and a graph $G$ let $\col_G(t)$ be the maximum number of edges 
of a $t$-colorable subgraph of~$G$, i.e., the size of the maximum $t$-cut in~$G$. For simplicity
we write $\col_n(t)$ for $\col_{K_n}(t)$. Tur\'an's theorem establishes
\[
	\ex_n(K_{t+1})=\col_n(t)\,.
\]
Babai, Simonovits, and Spencer~\cite{BSS90} were the first to investigate the extent to which such an identity 
can be extended to random graphs. In particular, those authors showed that 
it holds for $G(n,1/2)$ in the case of triangles ($t=2$), by showing that a.a.s.\ $G\in G(n,1/2)$ satisfies
\begin{equation}\label{eq:probMantel}
	\ex_G(K_3)=\col_G(2)\,.
\end{equation}
Answering a question from~\cite{BSS90}, it was shown by Brightwell, Panagiotou, and Steger~\cite{BPS12} that 
$p=1/2$ can be replaced by $p=n^{-\eta}$ for some $\eta>0$. Moreover, their proof extends to cliques of arbitrary fixed 
size and establishes that the identity $ex_G(K_{t+1})=\col_G(t)$ holds a.a.s.\ for $G\in G(n,p)$ as long as 
$p>n^{-\eta_t}$ for some sufficiently small $\eta_t>0$. Those authors conjectured that this result can be extended
to smaller values of~$p$. Note that~\eqref{eq:probMantel} holds trivially for very small $p$,
when a.a.s.\ the random graph itself is bipartite. However, here and below we shall exclude this range of~$p$.
It was noted in~\cite{BPS12} that (with the exception of small~$p$) 
in order for~\eqref{eq:probMantel} to hold, $p>c (\log n/n)^{1/2}$
is a necessary condition for some sufficiently small $c>0$. The reason for this is that 
for~$p<c (\log n/n)^{1/2}$, cycles of length five appear in $G(n,p)$ which have the additional property that none of its edges is 
contained in a triangle. Recently, DeMarco and Kahn~\cite{DMK} obtained a matching upper bound
by proving the following probabilistic version of Mantel's theorem (Theorem~\ref{thm:Turan} for $t=2$). 
\begin{theorem}
	\label{thm:probMantel}
	There exists a constant $C>0$ such that for $p>C(\log n /n)^{1/2}$
	a.a.s.\ $G\in G(n,p)$ satisfies $\ex_G(K_3)=\col_G(2)$. Moreover, every triangle-free subgraph 
	of~$G$ with the maximum number of edges is bipartite. 
\end{theorem}

It would be interesting to generalize this results to larger cliques. It seems plausible that a
necessary condition on~$p$ for such a generalization should come from the requirement that \emph{all} edges 
of $G(n,p)$ are contained in a cliques of size~$t+1$. In particular, the edges not contained in 
a copy of $K_{t+1}$ should not form a high chromatic 
subgraph.
For this we require on average 
$\Omega(\log n)$ such cliques per edge, instead of a constant number of cliques per edge, which gave rise to the $2$-density. 
For $K_{t+1}$ we get
\[
	p^{\binom{t+1}{2}}n^{t+1} = \Theta(pn^2\log n)\,.
\]
Solving this for $p$ leads to the following conjecture, which was stated by 
DeMarco and Kahn~\cite{DMK}.
\begin{conjecture}\label{conj:DMK}
	For every integer $t\geq2$ there exists a $C>0$ such that for $p\geq C((\log n)^\frac{1}{t-1}/n)^{\frac{2}{t+2}}$
	a.a.s.\ $G\in G(n,p)$ satisfies $\ex_G(K_{t+1})=\col_G(t)$.
\end{conjecture}
It would be also of interest to prove similar results for graphs~$F$ containing a color-critical edge. Partial 
results in this direction can be found in~\cite{BSS90} (see also~\cite{BPS12}).

\subsection{Triangle-free Graphs with Given Number of Vertices and Edges}
\label{sec:probKPR}
In this section we discuss extensions of Theorems~\ref{thm:KPR} and~\ref{thm:EFR}. Most of the work studied $\Forb_{n,M}(K_3)$, the set 
of triangle-free graphs with~$n$ vertices and~$M$ edges. 
The first result in this direction is due to Pr\"omel and Steger~\cite{PS96}, who proved a strengthening of the Erd\H os--Kleitman--Rothschild theorem (Theorem~\ref{thm:KPR} for $t=2$). It was shown that for $M>Cn^{7/4}\log n$, almost every graph $H\in \Forb_{n,M}(K_3)$ 
is bipartite. Similarly to the case of Tur\'an's theorem for random graphs discussed in the last section, such an assertion holds 
also for very small values of $M$, but not in the medium range (see Theorem~\ref{thm:OPT} below).
It was also noted in~\cite{PS96} that the statement fails to be true if $M=cn^{3/2}$ for some $c>0$.
The gap between $cn^{3/2}$ and $Cn^{7/4}\log n$ was closed by Osthus, Pr\"omel, and Taraz~\cite{OPT03} 
(see also Steger~\cite{S05} for a bit weaker result). In particular, the following result was shown in~\cite{OPT03}.
For positive integers $n$, $M$, and $t$, we denote by $\Col_{n,M}(t)$ the set of $t$-colorable graphs with $n$ vertices and 
$M$ edges.
\begin{theorem}
\label{thm:OPT}
For every $\eps>0$ the following holds
\[
	\lim_{n\to \infty}\frac{|\Forb_{n,M}(K_3)|}{|\Col_{n,M}(2)|} = \begin{cases}
		1, & \text{if $M=M(n)=o(n)$,}\\
		0, & \text{if $n/2\leq M=M(n)\leq (1-\eps)\frac{\sqrt{3}}{4}n^{3/2}\sqrt{\ln n}$,}\\
		1, & \text{if $M=M(n)\geq (1+\eps)\frac{\sqrt{3}}{4}n^{3/2}\sqrt{\ln n}$.}
	\end{cases}
\]
\end{theorem}
Note that similarly to Theorem~\ref{thm:probMantel} the ``critical window'' in Theorem~\ref{thm:OPT} concerns graphs 
with $\Theta(n^{3/2}\sqrt{\log n})$ edges. That might not be a coincidence, 
since having the property that every 
pair is covered by a path of length two seems to be a necessary condition for both problems. Generalizing this to the property that adding an edge 
for any pair of vertices would close a copy of $K_{t+1}$ 
suggests a joint generalization of the Kolaitis--Pr\"omel--Rothschild theorem, Theorem~\ref{thm:KPR}, and of 
Theorem~\ref{thm:OPT}, which was recently obtained by Balogh, Morris, Samotij, and Warnke~\cite{BMSW}.
%\begin{question}
%\label{quest:OPT} Is there a $C>0$ such that for $M=M(n)>C(\log n)^\frac{1}{t-1}n^{t+1})^{\frac{2}{t+2}}$
%we have $\lim_{n\to \infty}|\Forb_{n,M}(K_{t+1})|/|\Col_{n,M}(t)|=1$?
%\end{question}

A closely related result was proved by \L uczak. In~\cite{Lu00} he studied slightly sparser triangle-free graphs
and showed that
for $M=M(n)\gg n^{3/2}$, almost every graph $H\in \Forb_{n,M}(K_3)$ is ``close'' to a bipartite graphs,
i.e., it 
can be made bipartite by removing at most $o(M)$ edges. In fact, he also proved that this result generalizes for larger cliques, provided 
Conjecture~\ref{conj:KLR} (stated below), which we discuss in the next section, holds. 
Recently Balogh, Morris, and Samotij~\cite{BMS} and Saxton and Thomason~\cite{ST} developed an approach which, 
among other results, allowed them to prove Conjecture~\ref{conj:KLR} and
it could be used to verify \L uczak's statement directly.
\begin{theorem}\label{thm:Lu}
	For every  $\delta>0$ and $t\geq 2$ there exists a $C>0$ and $n_0$ such that for $M=M(n)\geq Cn^{2-1/m_2(K_{t+1})}$
	almost every graph $H$ drawn uniformly at random from $\Forb_{n,M}(K_{t+1})$ can be made $t$-colorable by removing at most 
	$\delta M$ edges.
\end{theorem}
It is known that, up to the constant~$C$, this result is best possible, and we also remark that the smallest~$M=M(n)$ 
in Theorem~\ref{thm:Lu} coincides in order of magnitude with the expected 
number of edges in $G(n,p)$ around the thresholds from Theorem~\ref{thm:probES}. 

\section{Regularity Method}
\label{sec:KLR}
One of the most important tools in extremal graph theory is Szemer\'edi's regularity lemma~\cite{Sz78},
and for a thorough discussion of its history and many of its applications we refer to~\cite{KSSS02,KS96}.
In fact, there were some applications of this lemma addressing extremal and Ramsey-type questions of 
random graphs (see, e.g.,~\cite{BSS90,RR95}). However, for the systematic study of extremal problems 
of $G(n,p)$ for $p=o(1)$, a variant of the lemma discovered independently by Kohayakawa~\cite{Ko97} and R\"odl (unpublished)
seemed to be an appropriate tool. We begin the discussion with Szemer\'edi's regularity lemma.

\subsection{Szemer\'edi's Regularity Lemma}
We first introduce the necessary definitions. Let $H=(V,E)$ be a graph, and let $X$, $Y\subseteq V$ be 
a pair of non-empty and disjoint subsets of the vertices.  We denote by $e_H(X,Y)$ the number of edges in the bipartite subgraph 
induced by $X$ and $Y$, i.e.,
\[
	e_H(X,Y)=\big|\big\{\{x,y\}\in E\colond x\in X\tand y\in Y\big\}\big|\,.
\]
We also define the \emph{density of the pair} $(X,Y)$ by setting 
\[
	d_H(X,Y)=\frac{e_H(X,Y)}{|X||Y|}\,.
\]
Moreover, we say a the pair $(X,Y)$ is \emph{$\eps$-regular} for some $\eps>0$, if  
\[
	|d_H(X,Y)-d_H(X',Y')|<\eps
\]
for all subsets $X'\subseteq X$ and $Y'\subseteq Y$ with $|X'|\geq \eps|X|$ and $|Y'|\geq \eps|Y|$. With this notation 
we can formulate Szemer\'edi's regularity lemma from~\cite{Sz78}.
\begin{theorem}[Regularity lemma]
	\label{thm:SzRL}
	For every $\eps>0$ and $t_0\in\NN$ there exist integers $T_0$ and $n_0$ such that every graph $H=(V,E)$
	with $|V|=n\geq n_0$ vertices admits a partition $V=V_1\dcup\dots\dcup V_t$ satisfying
	\begin{enumerate}[label=\rmlabel]
	\item \label{it:Sz1} $t_0\leq t\leq T_0$,
	\item \label{it:Sz2} $|V_1|\leq \dots \leq |V_t|\leq |V_1|+1$, and 
	\item \label{it:Sz3} all but at most $\eps t^2$ pairs $(V_i,V_j)$ with $i\neq j$ are $\eps$-regular.
	\end{enumerate}
\end{theorem}
Note that most applications of Theorem~\ref{thm:SzRL} involve dense graphs (i.e., $n$-vertex graphs with $\Omega(n^2)$ edges).
For each graph the lemma allows us to decompose the graph into bipartite ``blocks,'' the majority of which have a 
uniform edge distribution.
If such a graph has only $o(n^2)$ edges, it may not provide such control, since all edges may be contained in exceptional pairs 
(see property~\ref{it:Sz3} in Theorem~\ref{thm:SzRL}). Moreover, even for $\eps$-regular pairs, we do not gain any information if the 
density of that pair is~$o(1)$.

The following well known fact is used in many applications of the regularity lemma (see, e.g.,~\cite{KS96,RSch10}).
For future reference, we state both the \emph{embedding lemma} and the \emph{counting lemma}, even though the
latter clearly implies the former.

\begin{fact}[Embedding and counting lemma for dense graphs]\label{fact:cnt}
	For every graph $F$ with $V(F)=[\l]$ and every $d>0$, there exist $\eps>0$ and $m_0$
	such that the following holds.
	
	Let $H=(V_1\dcup\dots\dcup V_\l,E_H)$ be an $\l$-partite graph with $|V_1|=\dots=|V_\l|=m\geq m_0$ and
	with the property that for every edge $\{i,j\}\in E(F)$ the pair $(V_i,V_j)$ is $\eps$-regular in $H$ 
	with density $d_H(V_i,V_j)\geq d$.
	\begin{description}
	\item[Embedding Lemma]
	  Then $H$ contains a partite copy of~$F$, i.e., there exists a graph 
	homomorphism $\phi\colond F\to H$ with $\phi(i)\in V_i$.
	\item[Counting Lemma] The number of partite copies satisfies 
	\begin{multline}\label{eq:cnt}
		\big|\big\{\phi\colond F\to H\colond \phi\ \text{is a graph homomorphism with}\ \phi(i)\in V_i\big\}\big|\\
			=(1\pm f(\eps))\prod_{\{i,j\}\in E(F)} d(V_i,V_j)\prod_{i=1}^\l|V_i|\,,
	\end{multline}
	where $f(\eps)\to 0$ as $\eps\to 0$. 
	\end{description}
\end{fact}

As mentioned, the counting lemma implies the embedding lemma from Fact~\ref{fact:cnt}. However, for quite a few applications 
the existence of one copy is sufficient.

\subsection{Sparse Regularity Lemma for Subgraphs of Random Graphs}
\label{sec:spRL}
In this section we state a modified version of Szemer\'edi's regularity lemma, which allows applications to 
sparse graphs. Though more general lemmas are known, we restrict ourselves to a version which 
applies a.a.s.\ to all subgraphs of a random graph $G\in G(n,p)$. For that we first strengthen the notion 
of an $\eps$-regular pair.
\begin{definition}[$(\eps,p)$-regular pair]
	\label{def:ep-reg}
	Let $\eps>0$, let $p\in(0,1]$, let $H=(V,E)$ be a graph, 
	and let $X$, $Y\subseteq V$ be non-empty and disjoint. We say the pair $(X,Y)$ is \emph{$(\eps,p)$-regular}
	if 
	\[
		|d_H(X,Y)-d_H(X',Y')|<\eps p
	\]
	for all subsets $X'\subseteq X$ and $Y'\subseteq Y$ with $|X'|\geq \eps|X|$ and $|Y'|\geq \eps|Y|$.
\end{definition}
Note that $\eps$-regularity coincides with the case $p=1$ in the definition above.
However, for $p=p(n)=o(1)$ and graphs of density $\Omega(p)$ the notion of $(\eps,p)$-regularity gives additional control 
and addresses the second concern discussed after Theorem~\ref{thm:SzRL}. The sparse regularity lemma for subgraphs of $G(n,p)$
stated below asserts that, for those graphs $\eps$-regularity in Theorem~\ref{thm:SzRL} can be replaced by $(\eps,p)$-regularity.
In fact, besides the restriction to subgraphs of $G(n,p)$, this is the only difference between the following version of the sparse regularity lemma from~\cite{Ko97} and  Theorem~\ref{thm:SzRL}.
\begin{theorem}[Sparse regularity lemma for subgraphs of $G(n,p)$]
	\label{thm:SzRLGnp}
	For every $\eps>0$, $t_0\in\NN$, and every function $p=p(n)\gg 1/n$ there exist integers $T_0$ such that 
	a.a.s.\ $G\in G(n,p)$ has the following property. Every subgraph graph $H=(V,E)$ of $G$
	with $|V|=n$ vertices admits a partition $V=V_1\dcup\dots\dcup V_t$ satisfying
	\begin{enumerate}[label=\rmlabel]
	\item \label{it:Sz1Gnp} $t_0\leq t\leq T_0$,
	\item \label{it:Sz2Gnp} $|V_1|\leq \dots \leq |V_t|\leq |V_1|+1$, and 
	\item \label{it:Sz3Gnp} all but at most $\eps t^2$ pairs $(V_i,V_j)$ with $i\neq j$ are $(\eps,p)$-regular.
	\end{enumerate}
\end{theorem}

In order to make Theorem~\ref{thm:SzRLGnp} applicable in a similar way to Szemer\'edi's regularity lemma, one needs 
extensions of Fact~\ref{fact:cnt}. Theorem~\ref{thm:SzRLGnp} can be proved like the original regularity lemma 
with fairly straightforward adjustments. To prove a corresponding form of Fact~\ref{fact:cnt} turns out to be a challenging problem, which 
was resolved only recently in~\cite{BMS,CGSS,ST}. In particular, in the work of Balogh, Morris, and Samotij~\cite{BMS} 
and of Saxton and Thomason~\cite{ST}, a conjecture of Kohayakawa, \L uczak, and R\"odl~\cite{KLR97} was addressed. This conjecture implies a version of the embedding lemma of 
Fact~\ref{fact:cnt} appropriate for applications of Theorem~\ref{thm:SzRLGnp}. In~\cite{CGSS} only such a version was derived (see Theorem~\ref{thm:embGnp} below). 
For the formulation of the conjecture from~\cite{KLR97}, we require some more notation.
\begin{definition}\label{def:KLR}
	Let $\eps>0$, $p\in(0,1]$, $d>0$  and let $\l$, $m$, $M$ be integers. Let $F$ be a graph with vertex set 
	$V(F)=[\l]$. We denote by $\cG(F,m,M,\eps,p,d)$ the set of all $\l$-partite graphs $H=(V_1\dcup\dots\dcup V_\l,E_H)$ with
	\begin{enumerate}[label=\rmlabel]
		\item\label{it:KLR1} $|V_1|=\dots=|V_\l|=m$,
		\item\label{it:KLR2} $e_H(V_i,V_j)=M\geq dpm^2$ for all $\{i,j\}\in E(F)$, and
		\item $(V_i,V_j)$ is $(\eps,p)$-regular for all $\{i,j\}\in E(F)$\,.
	\end{enumerate}
	We denote by $\cB(F,m,M,\eps,p,d)$ the set of all those graphs from $\cG(F,m,M,\eps,p,d)$, which 
	contain no (partite) copy of $F$, i.e.,
	\begin{multline*}
		\cB(F,m,M,\eps,p,d)=\{H\in\cG(F,m,M,\eps,p,d)\colond \text{there is no} \\
			\text{graph homomorphism $\phi\colond F\to H$ with $\phi(i)\in V_i$}\}\,.
	\end{multline*}
\end{definition}
The first part of Fact~\ref{fact:cnt} asserts that for $p=1$, sufficiently small $\eps=\eps(F,d)>0$ and sufficiently large $m=m(F,d)$,
the set $\cB(F,m,M,\eps,p,d)$ is empty. However, if $p=o(1)$ then $\cB(F,m,M,\eps,p,d)$ is not empty for graphs~$F$ containing a cycle. 
In other words, if $p=o(1)$, then the regularity condition does not ensure the occurrences of copies of~$F$.
This prohibits a straightforward extension of Fact~\ref{fact:cnt} for the sparse regularity lemma. 
For example, as noted earlier for $p\ll n^{-1/m_2(F)}$ a.a.s.\ the random graph $G(n,p)$ contains only $o(pn^2)$
copies of some subgraph $F'\subseteq F$. Therefore, a.a.s.\ $G(n,p)$ contains an $F$-free subgraph with $(p-o(1))\binom{n}{2}$ edges.
This can be used to construct many $F$-free graphs $H\in\cG(F,m,M,\eps,p,d)$ for any $p=o(1)$ and appropriate choices of
$m$, $M$, and $d$. (For details see the discussion below Conjecture~\ref{conj:KLR}.) 
On the other hand, for 
$p\geq Cm^{-1/m_2(F)}$ for sufficiently large $C>0$, it was conjectured by Kohayakawa, \L uczak, and R\"odl in~\cite{KLR97} that $\cB(F,m,M,\eps,d)$ contains 
only ``very few'' graphs.

\begin{conjecture}[Kohayakawa, \L uczak \& R\"odl 1997]
	\label{conj:KLR}
	For every $\alpha>0$, $d>0$, and every graph $F$ with vertex set $V(F)=[\l]$, there are $\eps>0$, $C>0$
	and $m_0$ such that for every $m\geq m_0$, $p\geq Cm^{-1/m_2(F)}$ and $M\geq dpm^2$ we have 
	\begin{equation}\label{eq:KLR}
		|\cB(F,m,M,\eps,p,d)| \leq \alpha^M|\cG(F,m,M,\eps,p,d)|\,.
	\end{equation}
\end{conjecture}
Next we show that the lower bound on $p$ in Conjecture~\ref{conj:KLR} is necessary. 
For this, let $p=\delta m^{-1/m_2(F)}$ for some $\delta$ tending to $0$ with $m$.
We consider the family of graphs $\tcG(F,m,p,d)$ satisfying only 
properties~\ref{it:KLR1} and~\ref{it:KLR2} of Definition~\ref{def:KLR}, with $M=dpm^2$.
It is not hard to show that for every $\eps>0$ almost every $H\in \tcG(F,m,M,p,d)$ is also contained in
$\cG(F,m,M,\eps,p,d)$, i.e., 
\begin{equation}\label{eq:lowerKLR}
	|\cG(F,m,M,\eps,p,d)|\geq (1-o(1))|\tcG(F,m,p,d)|\,.
\end{equation}
Moreover, let $F'\subseteq F$ be the subgraph with $d_2(F')=m_2(F')$ (see~\eqref{eq:m2}), and let $e$ and $v$ denote its number of
edges and vertices, respectively. 
The expected number of partite copies of $F'$ in a graph $H$ chosen uniformly at random from $\tcG(F,m,p,d)$
is 
\[
	O((dp)^em^{v})=O((\delta d)^epm^2)=o(pm^2)\,.
\]
Hence, all but $o(|\tcG(F,m,p,d)|)$ graphs $H\in \tcG(F,m,p,d)$
have the property that, only $o(pm^2)$ edges of $H$ are contained in 
a copy of $F'$, and consequently also in a copy of $F$. 
Delete from each 
such~$H\in \cG(F,m,M,\eps,p,d)$
the edges contained in copies of $F$
and possibly a few more from each pair $(V_i,V_j)$, so that 
the resulting graph has precisely $M'=d'pm^2=(1-o(1)M$ edges for each such pair.
This way we obtain a graph $H'\in \cG(F,m,M',\eps',p,d')$ with $\eps'=\eps+o(1)$,
 which is $F$-free, i.e., $H'$ is contained in 
$\cB(F,m,M',\eps',p,d')$.
Consequently, we have 
\begin{align*}
	|\cB(F,m,M',\eps',p,d')|
	&\overset{\phantom{\eqref{eq:lowerKLR}}}{\geq} 
	(1-o(1))\frac{|\tcG(F,m,p,d)|}{\binom{M}{o(M)}^{e(F)}}\\
	&\overset{\phantom{\eqref{eq:lowerKLR}}}{=}
	(1-o(1))\left(\frac{\binom{m^2}{M}}{\binom{M}{o(M)}}\right)^{e(F)}\\
	&\overset{\phantom{\eqref{eq:lowerKLR}}}{\geq} (1-o(1))^{M'}\binom{m^2}{M'}^{e(F)} \\
	&\overset{\phantom{\eqref{eq:lowerKLR}}}{=}
	(1-o(1))^{M'}|\tcG(F,m,p,d')|\\
	&\overset{\phantom{\eqref{eq:lowerKLR}}}{\geq}
	 (1-o(1))^{M'}|\cG(F,m,M',\eps',p,d')|\,,
\end{align*}
which shows that~\eqref{eq:KLR} fails for $p=\delta n^{-1/m_2(F)}$
for sufficiently small $\delta>0$.

\subsection{Sparse Embedding and Counting Lemma} %for Subgraphs of Random Graphs}
%{Results related to Conjecture~\ref{conj:KLR}}
Conjecture~\ref{conj:KLR} is obvious, if $F$ is a matching. For all other graphs $F$, we have $m_2(F)\geq 1$, and 
the conjecture holds trivially for forests.  More interestingly, 
the conjecture was shown for cliques on at most six vertices~\cite{GPSST07,GSS04,KLR96} and (with an additional 
technical assumption) for cycles~\cite{KK97} (see also~\cite{Fu94} for an earlier related results for $F=C_4$).

Recently Conjecture~\ref{conj:KLR} was verified by Balogh, Morris, and Samotij~\cite{BMS} for 2-balanced graphs $F$ and 
by Saxton and Thomason~\cite{ST} for all graphs~$F$.  
\begin{theorem}\label{thm:KLR}
	Conjecture~\ref{conj:KLR} holds for all graphs~$F$.
\end{theorem}

One of the main motivations for the conjectured bound on the cardinality of 
$\cB(F,m,M,\eps,p,d)$ in~\eqref{eq:KLR} was that it easily implies that such ``bad''
graphs do not appear as subgraph of the random graph $G(n,p)$. In particular, we obtain 
an appropriate generalization of the embedding lemma from Fact~\ref{fact:cnt},
for subgraphs of $G(n,p)$ (see Theorem~\ref{thm:embGnp}). 
This result was also shown by Conlon, Gowers, Samotij, and
Schacht~\cite{CGSS} directly (without proving Conjecture~\ref{conj:KLR}).

\begin{theorem}[Embedding lemma for subgraphs of random graphs]
\label{thm:embGnp}
	For every graph $F$ with vertex set $V(F)=[\l]$ and every $d>0$ there exists $\eps>0$ such that for every 
	$\eta>0$ there exists $C>0$ such that for $p>Cn^{-1/m_2(F)}$ a.a.s.\ $G\in G(n,p)$ satisfies 
	the following.
	
	If $H=(V_1\dcup\dots\dcup V_\l,E_H)$ is an $\l$-partite (not necessarily induced) subgraph of~$G$ 
	with $|V_1|=\dots=|V_\l|\geq \eta n$ and
	with the property that for every edge $\{i,j\}\in E(F)$ the pair $(V_i,V_j)$ in $H$ is $(\eps,p)$-regular 
	and satisfies $d_H(V_i,V_j)\geq dp$, then $H$ contains a partite copy of~$F$, i.e., there exists a graph 
	homomorphism $\phi\colond F\to H$ with $\phi(i)\in V_i$.
\end{theorem}
\begin{proof}
We deduce Theorem~\ref{thm:embGnp} from Theorem~\ref{thm:KLR}. In fact, it will follow by a standard first 
moment argument. Since the result is trivial for matchings $F$ we may assume that $m_2(F)\geq 1$. 

For given $F$ and $d$ we set 
\[
	\alpha=\left(\frac{d}{2\euler}\right)^{e(F)}\,,
\]
where $\euler=2.7182\dots$ is the base of the natural logarithm.
Let $\eps'>0$ be given by the statement of Conjecture~\ref{conj:KLR}
applied with $F$, $d$, and~$\alpha$ and set $\eps=\eps'/2$. Following the quantification of Theorem~\ref{thm:embGnp},
we are given $\eta$. Finally, let $C'>0$ be given by 
Conjecture~\ref{conj:KLR} and set 
\[
	b=d\eta^2
	\qand
	C=\max\left\{\frac{C'}{\eta^{1/m_2(F)}}\,,\frac{\l}{b}\right\}
	\,. 
\]
Consider a graph $H'\subseteq G\in G(n,p)$ 
satisfying the assumptions of Theorem~\ref{thm:embGnp}. Let 
$m\geq \eta n$ be the size of the vertex classes, $V_1,\dots,V_{\l}$, and set
$M=dpm^2$. A straightforward  application of Chernoff's inequality asserts 
that $H'$ contains a spanning subgraph $H$ such that, for every $\{i,j\}\in E(F)$,
the pair $(V_i,V_j)$ is $(2\eps,p)$-regular, and $e_H(V_i,V_j)=M$. In other words, 
$H\in\cG(F,m,M,2\eps,p,d)$ and it suffices to show that a.a.s.\ $G\in G(n,p)$
contains no graph $H$ from $\cB(F,m,M,2\eps,p,d)$.

For that we consider the expected number of subgraphs in~$G$, which belong 
to $\cB(F,m,M,2\eps,p,d)$ for some $m\geq \eta n$. For $m\geq \eta n$
fixed, our choice of constants allows us to appeal to the conclusion 
of Theorem~\ref{thm:KLR}, and we obtain the following upper bound for 
the expected number of such graphs:
\begin{align*}
	p^{Me(F)}\cdot|\cB(F,m,M,2\eps,p,d)|\cdot\binom{n}{m}^\l
	&\leq
	p^{Me(F)}\cdot\alpha^M|\cG(F,m,M,2\eps,p,d)|\cdot\binom{n}{m}^\l\\
	&\leq 
	p^{Me(F)}\left(\frac{d}{2\euler}\right)^{Me(F)}\binom{m^2}{M}^{e(F)}2^{\l n}\\
	&\leq
	\left(p\cdot\frac{d}{2\euler}\cdot \frac{\euler}{pd}\right)^{Me(F)} 2^{\l n}\\
	&=
	2^{\l n-Me(F)}\\
	&\leq 
	2^{-bpn^2}\,,
\end{align*}
where we used  for the last estimate $M\geq dp(\eta n)^2$,
$e(F)\geq 2$, and $b=d\eta^2$ combined with
$\l n\leq bpn^2$ (which follows from $m_2(F)\geq 1$ and $C\geq \l/b$).

Summing the obtained bound over all possible values of~$m$ shows
that the expected number of bad graphs in $G$ is at most $n2^{-bpn^2}$,
and hence, Markov's inequality implies that a.a.s.\ $G\in G(n,p)$ contains no
such graph.
\end{proof}

Also the counting lemma of Fact~\ref{fact:cnt}
was partly extended to subgraphs in $G(n,p)$ in~\cite{CGSS}. We state these results below.
\begin{theorem}[Counting lemma for subgraphs of random graphs]
\label{thm:cntGnp}
	For every graph~$F$ with vertex set $V(F)=[\l]$ and every $d>0$ there exist $\eps>0$ and $\xi>0$ such that for every 
	$\eta>0$ there exists $C>0$ such that for $p>Cn^{-1/m_2(F)}$ a.a.s.\ $G\in G(n,p)$ satisfies 
	the following holds.
	
	Let $H=(V_1\dcup\dots\dcup V_\l,E_H)$ be an $\l$-partite (not necessarily induced) subgraph of~$G$ 
	with $|V_1|=\dots=|V_\l|\geq \eta n$ and
	with the property that for every edge $\{i,j\}\in E(F)$ the pair $(V_i,V_j)$ in $H$ is $(\eps,p)$-regular 
	with density $d_H(V_i,V_j)\geq dp$.

	\begin{enumerate}[label=\rmlabel]
	\item \label{it:cnt1Gnp} Then the number of partite copies of $F$ in $H$ is at least 	
		\begin{equation}
		\label{eq:cnt1Gnp}
			\xi p^{e(F)}\prod_{i=1}^\l|V_i|\,.
		\end{equation}		
	\item \label{it:cnt2Gnp}If in addition $F$ is strictly $2$-balanced, then the number of partite copies of $F$ in $H$
			satisfies
		\begin{equation}
		\label{eq:cnt2Gnp}
			 (1\pm f(\eps))p^{e(F)}\prod_{\{i,j\}\in E(F)} d(V_i,V_j)\prod_{i=1}^\l|V_i|\,,
		\end{equation}		
		where $f(\eps)\to 0$ as $\eps\to 0$.
	\end{enumerate} 
\end{theorem}

Let us briefly compare Theorems~\ref{thm:KLR}--\ref{thm:cntGnp}. Theorem~\ref{thm:KLR}, which
was proved in~\cite{ST}, gives an affirmative answer to Conjecture~\ref{conj:KLR} for all graphs~$F$,
and as we showed above, it implies Theorem~\ref{thm:embGnp}.  Also part~\ref{it:cnt1Gnp} of 
Theorem~\ref{thm:cntGnp} is a stronger version of Theorem~\ref{thm:embGnp}.
While Theorem~\ref{thm:embGnp} ensures only one copy of the given graph~$F$ in an appropriate 
$(\eps,p)$-regular environment, part~\ref{it:cnt1Gnp} of 
Theorem~\ref{thm:cntGnp} guarantees a constant fraction of the ``expected number'' of copies of $F$.
For strictly $2$-balanced graphs $F$, part~\ref{it:cnt2Gnp} of Theorem~\ref{thm:cntGnp}
guarantees the expected number of copies of~$F$, which can be viewed as the 
generalization of the counting lemma of Fact~\ref{fact:cnt} for such graphs~$F$.

Although Theorem~\ref{thm:embGnp} is the weakest result in this direction, 
it turns out to be sufficient for many natural applications of the regularity lemma or subgraphs of sparse 
random graphs (Theorem~\ref{thm:SzRLGnp}). For example, it allows new and conceptually simple proofs of 
Theorems~\ref{thm:probES} and~\ref{thm:prob-stability} (see, e.g., Section~\ref{sec:pf-prob-stability}
for such a proof of Theorem~\ref{thm:prob-stability}).

However, there are a few applications, where the full strength of Theorem~\ref{thm:KLR} was needed. 
For example, following the proof from~\cite{KK97} (see also~\cite{MSSS09}),
one can use the positive resolution of Conjecture~\ref{conj:KLR} to prove the $1$-statement of the threshold 
for the asymmetric Ramsey properties of random graphs (see Section~\ref{sec:aRamsey}), but Theorems~\ref{thm:embGnp} and~\ref{thm:cntGnp}
seem to be insufficient for this application.
In Section~\ref{sec:apps} we will also mention some applications, which require the quantitative 
estimates of Theorem~\ref{thm:cntGnp} (see Section~\ref{sec:removal} and~\ref{sec:Reiher}).

Finally, we remark that $G(n,p)$ has the properties of Theorem~\ref{thm:embGnp} and 
of part~\ref{it:cnt1Gnp} of Theorem~\ref{thm:cntGnp}
with probability $1-2^{-\Omega(pn^2)}$, while part~\ref{it:cnt2Gnp}
of Theorem~\ref{thm:cntGnp} holds with probability at least $1-n^{-k}$ for 
any constant $k$ and sufficiently large~$n$ (see~\cite{CGSS}). Also we note that, due to the 
upper bound on the number of copies of $F$ given in part~\ref{it:cnt2Gnp}
of Theorem~\ref{thm:cntGnp}, an error probability of the form $2^{-\Omega(pn^2)}$ 
can not hold. This is because, for $o(1)=p\gg 1/n$, the upper tail for the number of
copies of a graph $F$ (with at least as many edges as vertices) in $G(n,p)$
fails to have such a sharp concentration. In fact, the probability that $G(n,p)$ contains a clique
of size $2pn$ is at least $p^{\binom{2pn}{2}}=2^{-O(p^2\log(1/p) n^2)}\gg 2^{-\Omega(pn^2)}$,
and such a clique gives rise to $(2pn)^{|V(F)|}>2 p^{e(F)}n^{|V(F)|}$ 
copies of~$F$.

\section{Applications of the Regularity Method for Random Graphs}
\label{sec:apps}
In this section we show some examples how the regularity lemma and its counting and embedding lemmas 
for subgraphs of random graphs can be applied.

In Section~\ref{sec:aRamsey} we briefly review thresholds for asymmetric Ramsey properties 
of random graphs. In particular, Theorem~\ref{thm:KLR} can be used to establish the $1$-statement 
for such properties. We remark that, even though this is a statement about $G(n,p)$,
 in the proof suggested by Kohayakawa and Kreuter~\cite{KK97} one applies the sparse 
regularity lemma to an auxiliary subgraph of $G(n,p)$ with density $o(p)$. As a result 
Theorems~\ref{thm:embGnp} and~\ref{thm:cntGnp} cannot be applied anymore and an application
of Theorem~\ref{thm:KLR} is pivotal here.

In Section~\ref{sec:pf-prob-stability}, we transfer the Erd\H os--Simonovits theorem 
(Theorem~\ref{thm:stability}) to subgraphs of random graphs, i.e., we deduce  
Theorem~\ref{thm:prob-stability}. The proof given here is based on the 
sparse regularity lemma, and Theorem~\ref{thm:embGnp} suffices
for this application. It also utilizes the Erd\H os--Simonovits stability theorem, which will
be applied to the so-called \emph{reduced graph}.

In Section~\ref{sec:removal} we discuss another application and extend the \emph{removal lemma} (see 
Theorem~\ref{thm:removal} for the special case of triangles).
The standard proof of the removal lemma is based on Szemer\'edi's regularity lemma 
and the counting lemma of Fact~\ref{fact:cnt}. In fact, the embedding lemma seems not be sufficient 
for such a proof. The probabilistic version of the removal lemma for subgraphs of random graphs, 
Theorem~\ref{thm:prob-removal}, can be obtained by following the lines of the standard proof, where 
Szemer\'edi's regularity lemma and the counting lemma of Fact~\ref{fact:cnt} are replaced 
by the sparse regularity lemma (Theorem~\ref{thm:SzRLGnp}) and part~\ref{it:cnt1Gnp} of 
Theorem~\ref{thm:cntGnp}.

In Section~\ref{sec:Reiher} we state the recent \emph{clique density theorem} of Reiher~\cite{Rei} (see Theorem~\ref{thm:clique-density} below)
and its probabilistic version for random graphs. In the proof of the probabilistic version 
the ``right'' counting lemma (part~\ref{it:cnt2Gnp} of 
Theorem~\ref{thm:cntGnp}), giving the expected number of copies of cliques in an appropriate regular environment is an essential tool. Moreover, the clique density theorem itself will be applied to 
the \emph{weighted} reduced graph.

Finally in Section~\ref{sec:CGW} we briefly discuss some connection between the theory of quasi-random graphs 
and the regularity lemma. In particular, we will mention a generalization of a result of Simonovits and S\'os~\cite{SS97}
for subgraphs of random graphs and a strengthening of part~\ref{it:EN2} of Corollary~\ref{cor:EN}.

\subsection{Ramsey Properties of Random Graphs}
\label{sec:aRamsey}
Ramsey theory is another important field in discrete mathematics, 
which was influenced and shaped by Paul Erd\H os.
His seminal work with Szekeres~\cite{ErSz35} laid the ground 
for a lot of the research in   Ramsey theory. For example,
Graham, Spencer, and Rothschild~\cite[page~26]{GRS90} stated 
that,
\emph{``It is difficult to overestimate the effect of this paper.''}

For an integer $r\geq 2$ and graphs $F_1,\dots,F_r$, we denote by $\cR_n(F_1,\dots,F_r)$
the set of all $n$-vertex graphs $G$ with the Ramsey 
property, i.e., the $n$-vertex graphs $G$ with the property that for every $r$-coloring of the edges of $G$
with colors $1,\dots,r$ there exists a color $s$ such that~$G$ contains a copy of $F_s$ with all
edges colored with color~$s$. Ramsey's theorem~\cite{Ra30} implies that $\cR_n(F_1,\dots,F_r)$
is not empty for any $r$ and all graphs $F_1,\dots,F_r$ for sufficiently large~$n$.

While probabilistic techniques in Ramsey theory were introduced by Erd\H os~\cite{Er47}
in 1947, the investigation of Ramsey properties of the random graph $G(n,p)$ 
was initiated only in early 90's by \L uczak, Ruci\'nski, and Voigt~\cite{LRV92}.
In particular, one was interested in the threshold of $\cR_n(F_1,\dots,F_r)$ 
for the symmetric case, 
i.e., $F_1=\dots=F_r=F$, for which we use the short hand notation $\cR_n(F;r)$.
This question was addressed by R\"odl and Ruci\'nski~\cite{RR93,RR94,RR95}.
There it was shown 
that $n^{-1/m_2(F)}$ is the threshold for $\cR_n(F;r)$ for all graphs $F$ containing a cycle and all integers $r\geq 2$. Note that the threshold is independent of the number of colors~$r$.
The proof 
of the $1$-statement was based on
an application of Szemer\'edi's regularity lemma (Theorem~\ref{thm:SzRL}) for dense graphs, 
even though the result appeals to sparse random graphs.
Based on the recent embedding lemma for subgraphs of random graphs 
(Theorem~\ref{thm:embGnp}) and a standard application
of the sparse regularity
lemma (Theorem~\ref{thm:SzRLGnp}) a
conceptually simpler proof is now possible.

% A conceptually simpler 
%proof of this result follows from sparse regularity lemma (Theorem~\ref{thm:SzRLGnp}) and  
%its recently obtained embedding lemma (Theorem~\ref{thm:embGnp}).

Below we discuss the asymmetric Ramsey properties,
i.e., the case when not all~$F_i$ are the same graph. Here we restrict ourselves the 
the two-color case. Thresholds for asymmetric Ramsey properties involving cycles 
were obtained by Kohayakawa and Kreuter~\cite{KK97}.  
Furthermore, these authors
put forward a conjecture for the threshold of $\cR_n(F_1,F_2)$ 
for graphs~$F_1$ and~$F_2$ containing a cycle.
\begin{conjecture}\label{conj:KK} Let $F_1$ and $F_2$ be graphs containing a cycle
and $m_2(F_1)\leq m_2(F_2)$. Then $\ph=n^{-1/m_2(F_1,F_2)}$ is a threshold for $\cR_n(F_1,F_2)$,
where
\[
	m_2(F_1,F_2)
	=
	\max\left\{\frac{e(F')}{|V(F')|-2+1/m_2(F_1)}\colond F'\subseteq F_2\tand e(F')\geq 1\right\}\,.
\]
\end{conjecture}
There is an intuition behind the definition of $m_2(F_1,F_2)$, which has some analogy 
to the definition of $m_2(F)$ in~\eqref{eq:m2}. One can first observe that 
\[
	m_2(F,F)=m_2(F)
	\qand
	m_2(F_1)\leq m_2(F_1,F_2)\leq m_2(F_2)
	\,.
\]
Moreover, for $p\geq n^{-1/m_2(F_1,F_2)}$ the expected number of copies of $F_2$ (and all its subgraphs)
in $G(n,p)$ is of the same order of magnitude as the expected number of edges 
$G(n,n^{-1/m_2(F_1)})$. Assuming that there is a two-coloring of $G(n,p)$ with no copy of 
$F_2$ with edges in color two, one may hope that picking an edge of color one in every copy 
of $F_2$ may result in a graph with ``similar properties'' as $G(n,n^{-1/m_2(F_1)})$.
In particular, those edges should form a copy of~$F_1$ in color one.

In~\cite{KK97} the $1$-statement of Conjecture~\ref{conj:KK}
for $\cR_n(C,F)$ for any cycle~$C$ and any $2$-balanced graph~$F$ with $m_2(C)\geq m_2(F)$ 
was verified. Moreover, the $0$-statement was shown for the case when $F_1$ and $F_2$ are 
cliques~\cite{MSSS09}, and the $1$-statement was shown for graphs $F_1$ and $F_2$ with $m_2(F_1,F_2)>m_2(F_1,F')$ for every 
$F'\subsetneq F_2$ with $e(F')\geq 1$ appeared in~\cite{KSchS}. 
In particular, those results yield the threshold for $\cR(K_k,K_\l)$. 

It was also known that the resolution of Conjecture~\ref{conj:KLR} for the (sparser) graph $F_1$
allows us to generalize the proof from~\cite{KK97} to
verify the $1$-statement of Conjecture~\ref{conj:KK} 
when $F_2$ is strictly $2$-balanced (see, e.g.,~\cite{MSSS09}).
Therefore, Theorem~\ref{thm:KLR} has the following consequence.
\begin{theorem}
	Let $F_1$ and $F_2$ be graphs with 
	$1\leq m_2(F_1)\leq m_2(F_2)$ and let $F_2$ be strictly $2$-balanced.
	There exists a constant $C>0$ such that for $p\geq Cn^{-1/m_2(F_1,F_2)}$ a.a.s.\ $G\in G(n,p)$
	satisfies $G\in\cR_n(F_1,F_2)$.
\end{theorem}

\subsection{Stability Theorem for Subgraphs of Random Graphs}
\label{sec:pf-prob-stability}
Below we deduce a probabilistic version of the Erd\H os--Simonovits theorem from 
the classical stability theorem, based on the regularity method for subgraphs of random graphs.
\begin{proof}[Proof of Theorem~\ref{thm:prob-stability}]
	Let a graph $F$ with chromatic number $\chi(F)\geq 3$  and $\eps>0$ be given. In order to deliver the promised constants $C$ and $\delta$, we have to fix some auxiliary constants. First we
	appeal to the Erd\H os--Simonovits stability theorem, Theorem~\ref{thm:stability}, 
	with~$F$ and $\eps/8$ and obtain constants $\delta'>0$ and $n'_0$.
	Set
	\[
		\delta=\delta'/3\,.
	\]
	Moreover, set $d=\min\{\delta/4,\eps/4\}$ and set $\epsRL=\min\{\delta/8,\eps/8,\eps_{\textrm{EMB}}\}$, where $\eps_{\textrm{EMB}}$ is given by Theorem~\ref{thm:embGnp} applied with $F$ and $d$. Then apply 
	the sparse regularity lemma, Theorem~\ref{thm:SzRLGnp}, with $\epsRL$ and $t_0=\max\{n'_0, 4/\delta, 8/\eps\}$
	and obtain the constant~$T_0$. This gives us a lower bound of $n/T_0$ on the size of the partition classes
	after an application of Theorem~\ref{thm:SzRLGnp}. To a suitable collection of those classes, 
	we will want to apply Theorem~\ref{thm:embGnp}. Therefore, we set $\eta=1/T_0$.
	Due to our choice of $\epsRL\leq \eps_{\textrm{EMB}}$ Theorem~\ref{thm:embGnp}
	guarantees a constant $C=C(F,d,\epsRL,\eta)$ and we let $p\geq Cn^{-1/m_2(F)}$.
	
	For later reference we observe that, due this choice of constants above, for every $t\geq t_0$ we have 
	\begin{equation}\label{eq:pESaux}
		\frac{t}{2}+d\binom{t}{2}+\epsRL t^2 < \delta \binom{t}{2}
	\end{equation}
	and
	\begin{equation}\label{eq:pESaux2}
		\frac{1}{t}+\frac{d}{2}+\epsRL +\frac{\eps}{8} \leq \frac{\eps}{2}
		\,.
	\end{equation}

	We split the argument below into a probabilistic and a deterministic part. First, in the probabilistic part,
	we single out a few properties (see~\ref{it:pESf}--\ref{it:pESl} below), which the random graph 
	$G\in G(n,p)$ has a.a.s. In the second, deterministic part, we deduce the stability result for all
	graphs~$G$ satisfying those properties.

	In the probabilistic part we note that a.a.s.\ $G\in G(n,p)$ satisfies the following:	
	\begin{enumerate}[label=\alabel]
		\item\label{it:pESf} for all sets $X$, $Y\subseteq V(G)$ we have $e_G(X,Y)\leq (1+o(1))p|X||Y|$, 
			where the edges contained in $X\cap Y$ are counted twice,
		%\item\label{it:pESs0} $G$ satisfies the conclusion of Theorem~\ref{thm:probES}
		%	for $F$ and  $\eps$,
		\item\label{it:pESs} $G$ satisfies the conclusion of Theorem~\ref{thm:SzRLGnp} 
			for $\epsRL$, $t_0$, and $T_0$, 
		\item\label{it:pESl} $G$ satisfies the conclusion of Theorem~\ref{thm:embGnp} 
			for $F$, $d$, $\epsRL$, $\eta$ and $C$.
	\end{enumerate}
	Property~\ref{it:pESf} follows a.a.s.\ by a standard application of Chernoff's inequality, 
	and properties~\ref{it:pESs} and~\ref{it:pESl} hold a.a.s.\ due to Theorems~\ref{thm:SzRLGnp} 
	and~\ref{thm:embGnp}.
	
	In the deterministic part we deduce the conclusion of Theorem~\ref{thm:prob-stability} for all graphs 
	satisfying properties~\ref{it:pESf}--\ref{it:pESl}. To this end, let $G=(V,E)$ be a graph with these properties.
	Consider 
	an $F$-free subgraph $H\subseteq G$ with
	\[
		e(H)\geq \ex_G(F)-\delta p n^2\,.
	\]
	We will show that we can remove at most $\eps pn^2$ edges from $H$, so that the remaining 
	graph is $(\chi(F)-1)$-colorable. 
	
	Since every graph $G$ contains a $(\chi(F)-1)$-cut (see Definition~\ref{def:tcut})
	of size at least 
	\[
		\left(1-\frac{1}{\chi(F)-1}\right)e(G)=\pi(F)e(G)\,,
	\]
	it follows from property~\ref{it:pESf} that
	\begin{equation}
	\label{eq:pES-ES}
		e(H)\geq \pi(F)p\binom{n}{2}-2\delta p n^2\,.
	\end{equation} 
	We appeal to property~\ref{it:pESs}, which ensures the existence of
	a partition $V_1\dcup\dots\dcup V_t=V$ having properties~\ref{it:Sz1Gnp}--\ref{it:Sz3Gnp}
	of Theorem~\ref{thm:SzRLGnp} for $\epsRL$, $t_0$, and $T_0$. Without loss of generality, we may assume 
	that~$t$ divides~$n$ since removing at most $t$ vertices from~$H$ affects only $O(tn)=o(pn^2)$ edges.
	
	For the given partition, we consider the so-called \emph{reduced graph}
	$R=R(H,\epsRL,d)$ with vertex set $[t]$. The pair $\{i,j\}$ is an edge in $R$ if, and only if
	the pair $(V_i,V_j)$ is $(\epsRL,p)$-regular and $d_H(Vi,V_j)\geq dp$. 
	Note that $R$ does not represent the following edges of $H$:
	\begin{enumerate}[label=\RMlabel]
	\item\label{it:pES1} edges which are contained in some~$V_i$, 
	\item\label{it:pES2} edges which are contained in a pair $(V_i,V_j)$ which is not $(\eps,p)$-regular, and
	\item\label{it:pES3} edges which are contained in a pair $(V_i,V_j)$
	with $d_H(V_i,V_i)<dp$.
	\end{enumerate}
	Owing to property~\ref{it:pESf} we infer, that there are 
	at most 
	\begin{equation}\label{eq:pES1b}
		t\cdot (1+o(1))p\binom{n/t}{2}
	\end{equation}
	edges described in~\ref{it:pES1} and 
	at most 
	\begin{equation}\label{eq:pES2b}
		\epsRL t^2\cdot (1+o(1))p\left(\frac{n}{t}\right)^2
	\end{equation}
	edges described in~\ref{it:pES2}.
	By definition at most
	\begin{equation}\label{eq:pES3b}
		\binom{t}{2}\cdot dp\left(\frac{n}{t}\right)^2
	\end{equation}
	edges of $H$ are contained in pairs described in~\ref{it:pES3}.
	
	Moreover, since (again because of property~\ref{it:pESf}) 
	\[
		e_H(V_i,V_j)\leq e_G(V_i,V_j)\leq (1+o(1))p\left(\frac{n}{t}\right)
	\]
it follows from the definition of $R$, that 
	the number of edges in $R$ satisfies
	\begin{multline*}
	\label{eq:pES-eR}
		e(R)\geq \frac{e(H)
			-t\cdot (1+o(1))p\binom{n/t}{2}
			-\epsRL t^2\cdot (1+o(1))p(n/t)^2
			-\binom{t}{2}\cdot dp(n/t)^2}{(1+o(1))p(n/t)^2}\\
		\overset{\eqref{eq:pESaux},\eqref{eq:pES-ES}}{\geq}
		(\pi(F)-3\delta)\binom{t}{2}
		=
		(\pi(F)-\delta')\binom{t}{2}\,.
	\end{multline*}
	Moreover, property~\ref{it:pESl} implies that $R$ is $F$-free, since otherwise a copy of $F$ in $R$ 
	would lead to a copy of $F$ in $H$. In particular, $R$ satisfies the assumptions of the classical 
	Erd\H os--Simonovits stability theorem, Theorem~\ref{thm:stability}. Recall, that $\delta'>0$ was given 
	by an application of Theorem~\ref{thm:stability} applied with $F$ and $\eps/8$. We will only need the weaker assertion of Theorem~\ref{thm:stability}, which concerns the deletion of edges rather than the symmetric difference. 
	  Consequently, we may remove up to at most $(\epsSTAB/8) t^2$
	edges from~$R$, so that the resulting graph $R'$ is $(\chi(F)-1)$-colorable. Let 
	$f\colond [t]\to [\chi(F)-1]$ be such a coloring of $R'$
	and consider the corresponding partition $W_1\dcup\dots\dcup W_{\chi(F)-1}=V$ of $H$ given by
	\[
		W_i=\bigcup\big\{V_j\colond j\in f^{-1}(i)\big\}\,.
	\]
	It is left to show that 
	\[
		\sum_{i=1}^{\chi(F)-1}e_H(W_i)\leq \eps p n^2\,.
	\]
	Note that there besides the three types of edges described in~\ref{it:pES1}--\ref{it:pES3}
	the following type of edges of $H$ could be contained in $E_H(W_i)$ for some $i\in[\chi(F)-1]$
	\begin{enumerate}[label=\RMlabel]
	\setcounter{enumi}{3}
	\item\label{it:pES4} 
		edges which are contained in a pair $(V_i,V_j)$ for some $\{i,j\}\in E(R)\setminus E(R')$.
	\end{enumerate}
	Again property~\ref{it:pESf} combined with	
	$|E(R)\setminus E(R')|\leq (\eps/8)t^2$ implies that there 
	are at most 
	\begin{equation}\label{eq:pES4b}
		\frac{\eps}{8} t^2 \cdot (1+o(1))p\left(\frac{n}{t}\right)^2
	\end{equation}
	edges described in~\ref{it:pES4}.
	Finally, the desired bound follows from~\eqref{eq:pES1b}--\eqref{eq:pES4b} 
	\begin{align*}
		\sum_{i=1}^{\chi(F)-1}e_H(W_i)
		&\overset{\phantom{\eqref{eq:pESaux2}}}{\leq}
		t\cdot (1+o(1))p\binom{n/t}{2} \\
			&\qquad+\epsRL t^2\cdot (1+o(1))p\left(\frac{n}{t}\right)^2\\
			&\qquad+\binom{t}{2}\cdot dp\left(\frac{n}{t}\right)^2\\
			&\qquad+\frac{\eps}{8} t^2 \cdot (1+o(1))p\left(\frac{n}{t}\right)^2\\
			&\overset{\phantom{\eqref{eq:pESaux2}}}{\leq} (1+o(1))(1/t+d/2+\epsRL+\eps/8)pn^2\\
			&\overset{\eqref{eq:pESaux2}}{\leq} \eps p n^2\,.
	\end{align*}
	This concludes the proof of Theorem~\ref{thm:prob-stability}.
\end{proof}
We remark that there are several other classical results involving forbidden subgraphs $F$ ,
which can be transferred to subgraphs of random graphs, using a very similar approach, i.e., by applying the classical result to a suitably chosen reduced graph~$R$. For example, the $1$-statement 
of Theorem~\ref{thm:probES} or the $1$-statement of the Ramsey threshold from~\cite{RR95}
can be reproved by such an approach. In the next section, we discuss an example, where one 
can obtain the probabilistic result by ``repeating'' the original proof with the sparse regularity lemma
and a matching, embedding or counting lemma replacing Szemer\'edi's regularity lemma and 
Fact~\ref{fact:cnt}.

\subsection{Removal Lemma for Subgraphs of Random Graphs}
\label{sec:removal}
In one of the first applications of an earlier variant of the regularity lemma, Ruzsa and 
Szemer\'edi~\cite{RSz78} answered a question of Brown, S\'os, and Erd\H os~\cite{BES73} 
and essentially established the following removal lemma for triangles.
\begin{theorem}[Ruzsa \& Szemer\'edi, 1978]\label{thm:removal}
	For every $\eps>0$ there exist $\delta>0$ and $n_0$ such that every graph 
	$G=(V,E)$ with $|V|=n\geq n_0$ containing at most $\delta n^3$ copies of $K_3$
	can be made $K_3$-free by omission of at most $\eps n^2$ edges.
\end{theorem} 
In fact, the same statement holds, when~$K_3$ is replaced by any graph~$F$
and $\delta n^3$ is replaced by $\delta n^{|V(F)|}$, as was shown by F\"uredi~\cite{Fu95}
(see also~\cite{ADLRY94} for the case when $F$ is a clique and~\cite{EFR86} for related results). 
This result is now known as the \emph{removal lemma} 
for graphs (we refer to the recent survey of Conlon and Fox~\cite{CF}
for a thorough discussion of its importance and its generalizations).

The following probabilistic version for subgraphs of random graphs
was suggested by \L uczak~\cite{Lu06} and first proved for strictly $2$-balanced graphs~$F$
by Conlon and Gowers~\cite{CG}. The general statement for all $F$ follows from the work in~\cite{CGSS}.

\begin{theorem}
	\label{thm:prob-removal}
	For every graph~$F$ with $\l$ vertices 
	and $\eps>0$ there exist $\delta>0$ and $C>0$ such that for $p\geq Cn^{-1/m_2(F)}$ 
	a.a.s.\ for $G\in G(n,p)$ the following holds.	
	If $H\subseteq G$ contains at most $\delta p^{e(F)}n^\l$ copies of $F$, then 
	$H$ can be made $F$-free by omission of at most $\eps p n^2$ edges.
\end{theorem}
\begin{proof}
	Let a graph $F$ with $V(F)=[\l]$ vertices and $\eps>0$ be given. Since the result is trivial for matchings 
	$F$, we may assume that $m_2(F)\geq 1$.
	We will apply the counting lemma for subgraphs of $G(n,p)$ given by part~\ref{it:cnt1Gnp}
	of Theorem~\ref{thm:cntGnp}. We prepare for such an application with $F$ by setting $d=\eps/6$
	and choosing $\epsRL=\min\{\eps/6,\eps_{\textrm{CL}}/\l\}$, where $\eps_{\textrm{CL}}$
	is given by Theorem~\ref{thm:cntGnp}.
	Moreover, we set $t_0=3/\eps$ and let $T_0$ be given by the sparse regularity lemma, 
	Theorem~\ref{thm:SzRLGnp}, applied with~$\epsRL$ and~$t_0$. We then follow the quantification of
	Theorem~\ref{thm:cntGnp}. For that we set $\eta=(T_0\l)^{-1}$ and let $C>0$ be given 
	by Theorem~\ref{thm:cntGnp}. Finally, we set 
	\begin{equation}\label{eq:pRLdelta}
		\delta = \frac{1}{2}\xi p^{e(F)} \eta^\l\,.
	\end{equation}
	
	For later reference we observe that due this choice of constants above, for every $t\geq t_0$ and sufficiently large $n$ we have 
	\begin{equation}\label{eq:pRLaux}
		t\binom{n/t}{2}
		+2d\binom{t}{2}\left(\frac{n}{t}\right)^2
		+\epsRL n^2 \leq \frac{\eps}{2}n^2\,.
	\end{equation}
	
	Similarly, as in the proof given in Section~\ref{sec:pf-prob-stability}, we 
	split the argument in a probabilistic and a deterministic part. 	
	For the probabilistic part we note that a.a.s.\ $G\in G(n,p)$ satisfies the following:	
	\begin{enumerate}[label=\Alabel]
		\item\label{it:pRLf} for all sets $X$, $Y\subseteq V(G)$ we have $e_G(X,Y)\leq (1+o(1))p|X||Y|$, 
			where the edges contained in $X\cap Y$ are counted twice,
		%\item\label{it:pESs0} $G$ satisfies the conclusion of Theorem~\ref{thm:probES}
		%	for $F$ and  $\eps$,
		\item\label{it:pRLs} $G$ satisfies the conclusion of Theorem~\ref{thm:SzRLGnp} 
			for $\epsRL$, $t_0$, and $T_0$, 
		\item\label{it:pRLl} $G$ satisfies the conclusion of part~\ref{it:cnt1Gnp} of
			Theorem~\ref{thm:cntGnp} 
			for $F$, $d-\epsRL$, $\epsRL$, $\xi$, $\eta$, and~$C$.
	\end{enumerate}
	Again property~\ref{it:pRLf} follows a.a.s.\ by a standard application of Chernoff's inequality 
	and properties~\ref{it:pRLs} and~\ref{it:pRLl} hold a.a.s.\ due to Theorems~\ref{thm:SzRLGnp} 
	and~\ref{thm:cntGnp}.
	
	It is left to deduce the conclusion of Theorem~\ref{thm:prob-removal} for any graph 
	$G=(V,E)$ 
	satisfying properties~\ref{it:pRLf}--\ref{it:pRLl} and with sufficiently large $n=|V|$.
	Let $H\subseteq G$ containing at most $\delta p^{e(F)} n^{\l}$ copies of~$F$.

	Next we appeal to property~\ref{it:pRLs}, which ensures the existence of
	a partition $V_1\dcup\dots\dcup V_t=V$ having properties~\ref{it:Sz1Gnp}--\ref{it:Sz3Gnp}
	of Theorem~\ref{thm:SzRLGnp} for $\epsRL$, $t_0$, and $T_0$. Without loss of generality we may assume 
	that~$t\l$ divides~$n$ since removing at most~$t\l$ vertices from~$H$ affects only $O(t\l n)=o(pn^2)$ 
	edges.
	
	We remove the following edges from~$H$:
	\begin{itemize}
	\item edges which are contained in some~$V_i$, 
	\item edges which are contained in a pair $(V_i,V_j)$
	with $d_H(V_i,V_i)<2dp$,  and
	\item edges which are contained in a pair $(V_i,V_j)$ which is not $(\eps,p)$-regular.
	\end{itemize}
	Let $H'$ be the resulting subgraph.
	Owing to property~\ref{it:pRLf} we obtain
	\begin{align*}
		e(H)\setminus e(H')
		&\!\overset{\phantom{\eqref{eq:pRLaux}}}{\leq}
		t\cdot (1+o(1))p\binom{n/t}{2} %\\
			%&\qquad
			+\binom{t}{2}\cdot 2dp\left(\frac{n}{t}\right)^2%\\
			%&\qquad
			+\epsRL t^2\cdot (1+o(1))p\left(\frac{n}{t}\right)^2\\
			%&\overset{\phantom{\eqref{eq:pRLaux}}}{\leq} (1+o(1)(1/t+d/2+\epsRL)pn^2\\
			&\!\overset{\eqref{eq:pRLaux}}{\leq} (1+o(1))\frac{\eps}{2} p n^2\leq \eps pn^2\,.
	\end{align*}
	It is left to show that $H'$ is $F$-free. Suppose for a contradiction that $H'$
	contains a copy of~$F$. Let $V_{i_1},\dots,V_{i_k}$ be the vertex classes, that 
	contain a vertex from this copy. 
	
	Note that if $k=\l$, i.e., each class contains 
	exactly one vertex from~$F$, then the $\l$-partite induced subgraph 
	$H'[V_{i_1},\dots,V_{i_\l}]$ meets the assumptions of part~\ref{it:cnt1Gnp} 
	of Theorem~\ref{thm:cntGnp} for the constants $2d>d$, $\epsRL<\eps_{\textrm{CL}}$, $\xi$, $\eta$, and~$C$
	fixed above. Consequently, it follows from property~\ref{it:pRLl} that $H'$, and hence also $H$,
	contains at least
	\[
		\xi p^{e(F)} \left(\frac{n}{t}\right)^\l
		\geq 
		\xi p^{e(F)} (\eta\l)^\l\cdot n^\l 
		\overset{\eqref{eq:pRLdelta}}{>}
		\delta p^{e(F)}n^\l
	\]
 copies of~$F$, which contradicts the assumptions on~$H$.
 
 If $k<\l$, then subdivide every $V_{i_j}$ for $1\leq j\leq k$
 into $\l$ disjoint sets of size $|V_{i_j}|/\l$ in such a way that 
 every subclass created this way contains at most one vertex of the given 
 copy of~$F$ in~$H$. Let $W_1,\dots,W_\l$ be the classes containing one vertex of 
 the copy of~$F$ and we may assume that $W_i$ contains the copy of vertex~$i$
 of~$F$. Note that for every $i\in V(F)$ the set $W_i$ has size at least $n/(t\l)\geq \eta n$.
 Moreover, if $\{i,j\}\in E(F)$, then $H'[W_i,W_j]$ contains at least one edge. In particular,
 this edge is contained in $H'$ and, hence, it signifies that $(W_i,W_j)$ is 
 contained in some pair $(V_{i'}, V_{j'})$, which has density at least $dp$ and 
 which is $(\epsRL,p)$-regular. Moreover, it follows from the definition of $(\epsRL,p)$-regularity  
 (see Definition~\ref{def:ep-reg}), that $(W_i,W_j)$ is still $(\l\epsRL,p)$-regular and 
 has density at least $(d-\epsRL)p$. In other words, $H'[W_1,\dots,W_\l]$ is ready for an application
 of of Theorem~\ref{thm:cntGnp} for the constants $2d-\epsRL\geq d$, $\l\epsRL\leq \eps_{\textrm{CL}}$, 
 $\xi$, $\eta$, and~$C$
 fixed above. Consequently, it follows from property~\ref{it:pRLl} that $H'$, and hence also $H$,
 contains at least
	\[
		\xi p^{e(F)} \left(\frac{n}{t\l}\right)^\l
		\geq 
		\xi p^{e(F)} \eta^\l\cdot n^\l 
		\overset{\eqref{eq:pRLdelta}}{>}
		\delta p^{e(F)}n^\l
	\]
 copies of~$F$, which also in this case contradicts the assumptions on~$H$ and concludes the proof of Theorem~\ref{thm:prob-removal}.
\end{proof}

\subsection{Clique Density Theorem for Subgraphs of Random Graphs}
\label{sec:Reiher}
Tur\'an's theorem establishes the minimum number $\ex_n(K_{k})+1$
of edges in an $n$-vertex graph that implies the existence of a copy of $K_{k}$.
For the triangle case, it was proved by Hans Rademacher (unpublished) that every $n$-vertex graph 
with $\ex_n(K_{3})+1$ edges contains not only one, but at least $n/2$ triangles.
More generally, Erd\H os suggested to study the minimum number of triangles in 
$n$-vertex graphs with $\ex_n(K_3)+s$ edges~\cite{Er55,Er62}. He conjectured that 
for $s<n/4$ there are at least $s\lfloor n/2\rfloor$ triangles, which is best possible 
due to the graph obtained by balanced, complete bipartite graph with $s$ independent edges 
in the vertex  with $\lceil n/2\rceil$ vertices. This conjecture was proved by Lov\'asz and 
Simonovits~\cite{LS76} (see also~\cite{KhNi81}). 
For larger values of~$s$ and $k$ this problem was studied by Erd\H os~\cite{Er62c},
Moon and Moser~\cite{MM62}, Nordhaus and Stewart~\cite{NS63}, Bollob\'as~\cite{Bo76a,Bo76},
and Khadzhiivanov and Nikiforov~\cite{KhNi78}.

In particular, in~\cite{LS83} Lov\'asz and Simonovits formulated a conjecture which 
relates the minimum density of $K_{k}$ with a given edges density. More precisely, 
for an integer $k\geq 3$ and a graph $H$, let $\cK_k(H)$ be the number of (unlabeled) 
copies of~$K_k$ in $H$. We denote by $\cK_k(n,M)$ the minimum over all graphs
with $n$ vertices and $M$ edges, i.e., 
\[
	\cK_k(n,M)=\min\{\cK_k(H)\colond |V(H)|=n\tand |E(H)|=M\}\,.
\]
In~\cite{LS76} Lov\'asz and Simonovits conjectured that the extremal graph for 
$\cK_k(n,M)$ is obtained from complete $t$-partite graph (for some appropriate $t$)
by adding a matching to one of the vertex classes. In~\cite{LS83} those authors proposed 
an approximate version of this conjecture by considering densities of cliques and edges 
instead relating the number of cliques with the number of edges. For that we define for 
$\alpha\in[0,1)$
\[
	\kappa_k(\alpha)
	=
	\liminf_{n\to\infty} 
	\frac{\cK_k(n,\left\lceil\alpha\binom{n}{2}\right\rceil)}{\binom{n}{k}}\,,
\]
i.e., $n$-vertex graphs with $\alpha \binom{n}{2}$ edges contain at least 
$(\kappa_k(\alpha)-o(1))\binom{n}{k}$ copies of $K_k$ and $\kappa_k(\alpha)$ is the largest 
clique density which can be guaranteed. Clearly, $\kappa_k(\cdot)$ is non-decreasing
and for $\alpha\in[0,\pi(K_k))$ 
we have $\kappa_k(\alpha)=0$.
Lov\'asz and Simonovits suggested that the 
graphs described below attain the infimum of $\kappa_k(\alpha)$: For a given $\alpha>0$, let~$t$ be the
integer with the property 
\begin{equation}
\label{eq:Reiher-t}
	\alpha \in\left(1-\frac{1}{t},1-\frac{1}{t+1}\right]
\end{equation}
and set $\rho\in \RR$ to the smaller root of the quadratic equation
\[
	2\rho(1-\rho)+\left(1-\frac{1}{t}\right)(1-\rho)^2=\alpha\,.
\]
One can check that~\eqref{eq:Reiher-t} implies  $0<\rho\leq \frac{1}{t+1}$.
We then define the graphs~$T_{n,\alpha}$ to be the complete $(t+1)$-partite graph 
with vertex classes $V_1\dcup\dots\dcup V_{t+1}=V(T_{n,\alpha})$ satisfying
\[
	|V_{t+1}|=\lceil\rho n\rceil
	\qand
	|V_1|\leq\dots\leq|V_t|\leq |V_1|+1\,.
\] 
Maybe, a more intuitive description of these graphs is the following.
For edge densities $\alpha$ of the form $1-1/t$ the graph 
$T_{n,\alpha}$ is the Tur\'an graph $T_{n,t}$ with $t$ classes.
For $\alpha\in (1-\frac{1}{t},1-\frac{1}{t+1})$ a ``small'' $(t+1)$st class of size $\rho n$
appears and all other classes have size $(1-\rho)n/t$.
With $\alpha$ tending to $1-\frac{1}{t+1}$ the difference in size between the $(t+1)$st class
and the other classes becomes smaller. Finally, for $\alpha=1-\frac{1}{t+1}$ 
we get $\rho=1/(t+1)$ and $T_{n,\alpha}$
becomes the Tur\'an graph with $t+1$ classes.

Lov\'asz and Simonovits conjectured that for every $k$ and $\alpha$
\begin{equation}\label{eq:LSconj}
	\kappa_k(\alpha)=\lim_{n\to\infty}\frac{\cK_k(T_{n,\alpha})}{\binom{n}{k}}\,.
\end{equation}
We remark that the conjectured extremal graph $T_{n,\alpha}$ is independent of the size of clique~$K_k$.
This conjecture was known to be true in the ``symmetric case,'' i.e., 
for densities $\alpha\in\{1-1/t\colond t\in\NN\}$, due to the work of 
Moon and Moser~\cite{MM62} (see~\cite{NS63} for the triangle case).
Fisher addressed~\eqref{eq:LSconj} for $k=3$ and $1/2\leq \alpha\leq 2/3$. 

A few years ago Razborov introduced the so-called \emph{flag algebra method}
in extremal combinatorics~\cite{Ra07} (see~\cite{RaSurvey} for a survey on the topic)
and based on this calculus he solved the triangle 
case for every $\alpha\in(0,1)$ in~\cite{Ra08}. This work was followed by Nikiforov~\cite{Ni11}, which led to the solution 
of the case~$k=4$ and finally Reiher~\cite{Rei} verfied the conjecture for every $k$.
\begin{theorem}[Clique density theorem]
	\label{thm:clique-density}
	For every integer $k\geq 3$ and for  every $\alpha\in(0,1)$ we have
	\[
		\kappa_k(\alpha)=\lim_{n\to\infty}\frac{\cK_k(T_{n,\alpha})}{\binom{n}{k}}\,.
	\]
\end{theorem}

Based on the counting lemma for subgraphs of random graphs, part~\ref{it:cnt2Gnp} 
of Theorem~\ref{thm:cntGnp}, one can use the sparse regularity lemma 
to transfer this result to subgraphs of random graphs. The following appears in~\cite{CGSS}.
\begin{theorem}
	\label{thm:prob-clique-density}
	For every graph~$k\geq 3$ and $\delta>0$ there exists $C>0$ such that 
	for $p\geq Cn^{-1/m_2(K_k)}$ the following holds a.a.s.\ for $G\in G(n,p)$.	
	If $H\subseteq G$ contains at least $(\alpha+\delta)e(G)$ edges, then 
	\[
		\cK_k(H)\geq \kappa_k(\alpha) \cK_k(G)\,.
	\]	
\end{theorem}
As mentioned above, the proof of 
Theorem~\ref{thm:prob-clique-density} is based on the regularity method for subgraphs 
of random graphs and relies on the counting lemma giving the ``expected number'' of copies 
of $K_k$ in an appropriate $(\eps,p)$-regular environment. Moreover, in the proof 
a \emph{weighted version} of the clique density theorem, Theorem~\ref{thm:clique-density} 
is applied to the weighted reduced graph (see~\cite{CGSS} for details).

\subsection{Quasi-random Subgraphs of Random Graphs}
\label{sec:CGW}
In this section we discuss scaled versions of the Chung-Graham-Wilson theorem~\cite{CGW89}
on quasi-random graphs for subgraphs a random graphs.
The systematic study of quasi-random graphs was initiated by Thomason~\cite{T87a,T87b} and 
Chung, Graham, and Wilson~\cite{CGW89} (see also~\cite{AC88,FRW88,Ro86} for partial earlier 
results and~\cite{KS06} for a recent survey on the topic). 
In~\cite{CGW89} 
several properties of dense random graphs,
i.e., properties a.a.s.\ satisfied by $G(n,p)$ for $p>0$ independent of~$n$, 
were shown to be equivalent in a deterministic sense. This phenomenon fails to be true 
for $p=o(1)$ (see, e.g.,~\cite{CG02,KR03}).
For relatively dense subgraphs of sparse random graphs however 
several deterministic equivalences among (appropriately scaled)
quasi-random properties remain valid. %(see, e.g.,~\cite{KR03} for some results in this direction)
Below we will discuss one such equivalence (see Theorem~\ref{thm:probSiSo} below), whose analog for dense graphs 
was obtained by Simonovits and S\'os~\cite{SS97}. 

Before we mention the result of Simonovits and S\'os we begin with the
following quasi-random properties of graphs, concerning the edge distribution (see $\DISC$ below) 
and the number of copies (or embeddings) of given graph~$F$ (see $\CL$ below).
\begin{definition} Let $F$ be a graph 
on $\l$ vertices and let $d>0$.
\begin{description} 
	\item[$\DISC$] We say a graph $H=(V,E)$ with $|V|=n$ satisfies $\DISC(d)$, 
		if for every subset $U\subseteq V$ 
		we have 
		\[
			e_H(U)=d\binom{|U|}{2}\pm o(n^2)\,.
		\]
	\item[$\CL$] We say a graph $H=(V,E)$ with $|V|=n$ satisfies $\CL(F,d)$, if the number $N_F(H)$ of labeled 
	copies of $F$ in $H$ satisfies 
	\[
		N_F(H)=d^{e(F)}n^\l\pm o(n^{\l})\,.
	\]
\end{description}
\end{definition}
It is well known that the property $\DISC(d)$ implies 
the property $\CL(F,d)$ for every graph~$F$. By this we mean that for every $\eps>0$ 
there exist $\delta>0$ and $n_0$ such that every $n$-vertex  $(n\geq n_0)$
graph $H$ satisfying $\DISC(d)$ with $o(n^2)$ replaced by~$\delta n^2$ also 
satisfies property $\CL(F,d)$ with $o(n^{\l})$ replaced by $\eps n^\l$. 

The opposite implication is known to be false. For example, for $F=C_\l$ being a cycle 
of length $\l$ we may consider $n$-vertex graphs $H$ consisting of a clique of size $dn$ and isolated 
vertices. Such a graph $H$ satisfies $\CL(C_\l,d)$, but fails to have $\DISC(d)$.
However, note that such a graph $H$ fails to have density~$d$. If we add this as an additional
condition, then for even $\l$ one of the main implications of the Chung-Graham-Wilson theorem 
asserts the implication $\CL(C_\l,d)$ implies $\DISC(d)$.

For odd cycles imposing global density $d$ does not suffice, as the following 
interesting example from~\cite{CGW89} shows:
%In addition to $\CL(C,d)$ we also impose that~$H$ has edge density $d(H)$
%at least $d$, then $\CL(C,d)$ implies $\DISC(d)$ for cycles $C$ of even length.
%For odd cycles this additional assumption still does not suffice
%as the following example from~\cite{CGW89} shows: 
Partition the vertex set 
$V(H)=V_1\dcup V_2\dcup V_3\dcup V_4$ as equal as possible into four sets
and add the edges of the complete graph on $V_1$ and on $V_2$, add edges of the complete 
bipartite graph between $V_3$ and $V_4$, and add edges of a random bipartite graph 
with edge probability
$1/2$ between $V_1\dcup V_2$ and $V_3\dcup V_4$. One may check 
that a.a.s.\ such a graph $H$ has density $d(H)\geq 1/2-o(1)$ and it 
satisfies $\CL(K_3,1/2)$, but clearly it fails to have $\DISC(1/2)$.

Summarizing the discussion above, while $\DISC(d)$ implies $\CL(F,d)$ is known to be true for every 
graph~$F$, $\CL(C_{2\l+1},d)$ does not imply $\DISC(d)$. In fact, $\CL(C_{2\l+1},d)\nRightarrow\DISC(d)$
even when restricting to graphs $H$ with density~$d$. 
 
Note that the property $\DISC(d)$ is \emph{hereditary} in the sense that 
for subsets $U\subseteq V$ the induced subgraph $H[U]$ must also satisfy $\DISC(d)$, 
if~$H$ has~$\DISC(d)$. As a result the implication $\DISC(d)\Rightarrow\CL(F,d)$
extends to the following hereditary strengthening of $\CL(F,d)$.
\begin{description} 
	\item[$\HCL$] We say a graph $H=(V,E)$ with $|V|=n$ satisfies $\HCL(F,d)$, if for every $U\subseteq V$ the number $N_F(H[U])$ of labeled 
	copies of $F$ in the induced subgraph~$H[U]$ satisfies 
	\begin{equation}\label{eq:HCL}
		N_F(H[U])=d^{e(F)}|U|^\l\pm o(n^{\l})\,.
	\end{equation}
\end{description}
It was shown by Simonovits and S\'os~\cite{SS97} (see also~\cite{Y10} for a recent strengthening)
that $\HCL$ indeed is a quasi-random property, i.e., 
those authors showed that for every graph $F$ with at least one edge
and for every $d>0$ the properties $\DISC(d)$ and $\HCL(F,d)$ are equivalent, i.e., 
$\DISC(d)\Rightarrow\HCL(d)$ and $\HCL(d)\Rightarrow\DISC(d)$.

Based on the sparse regularity lemma (Theorem~\ref{thm:SzRLGnp}) and its appropriate counting lemma (part~\ref{it:cnt2Gnp} 
of Theorem~\ref{thm:cntGnp}) a generalization of the 
Simonovits--S\'os theorem for subgraphs of random graphs $G\in G(n,p)$ 
can be derived. First we introduce the appropriate sparse versions 
of $\DISC$ and~$\HCL$ for this context. 

\begin{definition} Let $G=(V,E)$ be a graph with $|V|=n$, let $F$ be a graph 
on $\l$ vertices, and let $d>0$ and $\eps>0$. 
\begin{description} 
	\item[$\DISC_G$] We say a subgraph $H\subseteq G$ satisfies $\DISC_G(d)$, if for every subset $U\subseteq V$ 
		we have 
		\[
			e_H(U)=d|E(G[U])|\pm o(|E|)\,,
		\]
		i.e., the relative density of $H[U]$ with respect to $G[U]$ is close to~$d$
		for all sets $U$ of linear size.
		Furthermore, we say $H$ satisfies $\DISC_G(d,\eps)$ if  $e_H(U)=d|E(G[U])|\pm \eps|E|$.
	\item[$\HCL_G$] We say a subgraph $H\subseteq G$ satisfies $\HCL_G(F,d)$, if for every 
		$U\subseteq V$ the number $N_F(H[U])$ of labeled 
		copies of $F$ in the induced subgraph $H[U]$ satisfies 
	\[
		N_F(H[U])=d^{e(F)}N_F(G[U])\pm o(|N_F(G)|)\,,
	\]
	i.e., approximately a $d^{e(F)}$ proportion of the copies of $F$ in $G[U]$ 
	is contained in~$H[U]$ for sets $U$ spanning a constant proportion of copies of $F$ in $G$.
	
	Furthermore, we say a subgraph $H\subseteq G$ satisfies $\HCL_G(F,d,\eps)$, if for every 
		$U\subseteq V$ we have  $N_F(H[U])=d^{e(F)}N_F(G[U])\pm \eps|N_F(G)|$.
	\end{description}
\end{definition}
For those properties one can prove an equivalence in the sense described above, when $G$ is a random graph.

\begin{theorem}\label{thm:probSiSo}
	Let $F$ be a strictly $2$-balanced graph with at least one edge and let $d>0$. 
	For every $\eps>0$ there exist $\delta>0$ and $C>0$ such that for  $p\geq Cn^{-1/m_2(F)}$,
	a.a.s.\ for $G\in G(n,p)$ the following holds.
	\begin{enumerate}[label=\rmlabel]
		\item If $H\subseteq G$ satisfies $\DISC_G(d,\delta)$, then $H$ satisfies $\HCL_G(F,d,\eps)$.
		\item If $H\subseteq G$ satisfies $\HCL_G(F,d,\delta)$, then $H$ satisfies $\DISC_G(d,\eps)$.
	\end{enumerate}
	Consequently, for $p\gg n^{-1/m_2(F)}$
	a.a.s.\ $G\in G(n,p)$ has the property that $\DISC_G(d)$ and $\HCL_G(F,d)$ are equivalent.
\end{theorem}

We will briefly sketch some ideas from the proofs of both implications of the Simonovits-S\'os theorem
and indicate its adjustments for the proof of  Theorem~\ref{thm:probSiSo}.
%below, before we state the 
%extension of the Simonovits--S\'os theorem to subgraphs of random graphs (see Theorem~\ref{thm:probSiSo} below).

The implication $\DISC(d)\Rightarrow\HCL(F,d)$  (for dense graphs)
easily follows from the counting lemma in Fact~\ref{fact:cnt}.
Indeed, given $U\subseteq V(H)$ for a graph $H$ satisfying $\DISC(d)$, we consider a partition of 
$U$ into $\l=|V(F)|$ classes $U_1\dcup\dots\dcup U_\l$
with sizes as equal as possible. Based on the identity
\[
	e_H(U_i,U_j)=e_H(U_i\dcup U_j)-e_H(U_i)-e_H(U_j)
\]
we infer from $\DISC(d)$ that $(U_i,U_j)$ is $g(\eps)$-regular and $d_H(U_i,U_j)=d\pm o(1)$, 
where $g(\eps)$ tends to 0 with the error parameter $\eps$ from the property $\DISC(d)$.
In particular, for sufficiently small $\eps>0$ the assumptions of the counting lemma of Fact~\ref{fact:cnt}
are met for~$F$ and~$d$. Using the upper and lower bound on the number of partite copies of~$F$ 
in $U_1\dcup\dots\dcup U_{\l}$ provided by the counting lemma 
and a simple averaging argument over all possible partitions $U_1\dcup\dots\dcup U_\l$ 
yields~\eqref{eq:HCL}. This simple argument with part~\ref{it:cnt2Gnp} 
of Theorem~\ref{thm:cntGnp} replacing Fact~\ref{fact:cnt} can be transferred to $G(n,p)$ without any further adjustments.

The proof of the opposite implication, $\HCL(F,d)\Rightarrow\DISC(d)$, is more involved. All known proofs for dense graphs 
are based on Szemer\'edi's regularity lemma (Theorem~\ref{thm:SzRL}) and the counting lemma in Fact~\ref{fact:cnt}.
The proof of Simonovits and S\'os requires not only
an applications of Fact~\ref{fact:cnt} for $F$, but also for a graph obtained from $F$ 
by taking two copies of~$F$ and identifying one of 
their edges. Note that this ``double-$F$'' is not strictly $2$-balanced, since it contains $F$ as a proper subgraph, 
which has the same $2$-density as double-$F$. Consequently, we run into some difficulties, if we want to extend 
this proof to subgraphs of random graphs based on part~\ref{it:cnt2Gnp} of Theorem~\ref{thm:cntGnp}.
In some recent generalizations of the Simonovits--S\'os theorem applications of Fact~\ref{fact:cnt} for double-$F$
could be avoided (see, e.g.,~\cite{S08,CHPS12,DR11}). In particular, the proof presented
in~\cite[pages 174-175]{DR11} extends to subgraphs of random graphs, by replacing Szemer\'edi's regularity lemma 
and the counting lemma of Fact~\ref{fact:cnt} by its counterparts for sparse random graphs.

\subsubsection{Problem of Erd\H os and Ne\v set\v ril revisited}
We close this section by returning to the question of Erd\H os and Ne\v set\v ril from 
Section~\ref{sec:EN}, which perhaps led  
to one of the first extremal results for random graphs.
 
Here we want to focus on generalizations of part~\ref{it:EN2} of Corollary~\ref{cor:EN}.
That statement asserts that any $K_{k+1}$-free graph $H$
with the additional property 
\begin{equation}\label{eq:ENGnp}
	\ex_H(K_k)\leq (\pi(K_k)+o(1))e(H)
\end{equation}
must have vanishing density $d(H)=o(1)$.
%Roughly speaking, in the proof of this assertion we used the assumption $\ex_G(K_k)\leq (\pi(K_k)+o(1))e(G)$
%to localize a large quasi-random subgraph of density $d'\geq d$ (see Lemma~\ref{lem:EN2}). 
%Owing to the implication $\DISC(d')\Rightarrow\CL(K_{k+1},d')$ this gives rise to many copies of $K_{k+1}$.
%This argument can be transferred to subgraphs of random graphs $G(n,p)$ as long as $p>Cn^{-1/m_2(K_{k+1})}$
%for sufficiently large~$C$. Therefore, one can now prove the following 
%strengthening of part~\ref{it:EN2} of Corollary~\ref{cor:EN}.
% and we sketch the details below.

Based on Theorem~\ref{thm:embGnp} the following generalization can be proved.
Consider the random graph $G\in G(n,p)$ for $p\gg n^{-1/m_2(K_{k+1})}$.
We will show that  
a.a.s.\ $G$ has the property that any $K_{k+1}$-free graph $H\subseteq G$ 
satisfying~\eqref{eq:ENGnp} must have vanishing \emph{relative} density (w.r.t.\ the density of $G$), i.e., $d(H)=o(p)$.
\begin{theorem}[Generalization of Corollary~\ref{cor:EN}\ref{it:EN2} for subgraphs of $G(n,p)$]
	\label{thm:probEN}
	For every integer $k\geq 3$, every $d>0$, and every $\eps\in(0,1-\pi(K_k))$
	there exists some $C>0$ such that for $p>Cn^{-1/m_2(K_{k+1})}$ the following holds a.a.s.\
	for $G\in G(n,p)$. If $H\subseteq G$ satisfies $e(H)=d|E(G)|$ and 
	$\ex_H(K_k)\leq(\pi(K_k)+\eps)e(H)$, then $H$ contains a $K_{k+1}$.
\end{theorem}
%\begin{proof}[Proof (sketch).]
%\end{proof}
The proof of Theorem~\ref{thm:probEN} follows the lines of the proof of Corollary~\ref{cor:EN}\ref{it:EN2}
given in Section~\ref{sec:EN} and we briefly sketch the main adjustments needed. 
Recall that the proof given in Section~\ref{sec:EN} relied on Lemma~\ref{lem:Ro86}
from~\cite{Ro86}, which is based on the embedding lemma of Fact~\ref{fact:cnt}. Replacing the embedding lemma for dense graphs
by the appropriate version for subgraphs of random graphs, i.e., by Theorem~\ref{thm:embGnp}, yields the following.
\begin{lemma}\label{lem:Ro86G}
	For all integers $s$, $t\geq 2$ and every $d>0$ there exist $\delta>0$ and $C>0$
	such that for $p>Cn^{-1/m_2(K_{s})}$ the following holds a.a.s.\
	for $G\in G(n,p)$.
	
	If $H\subseteq G$ satisfying $e_H(U)=(d\pm\delta)e_G(U)$
	for every $U\subseteq V$ with $|U|=\lfloor n/t\rfloor$. Then $H$ contains a copy of~$K_s$.
\end{lemma}
Equipped with Lemma~\ref{lem:Ro86G} one can  repeat the proof of Lemma~\ref{lem:EN2}
and the following appropriate version for subgraphs of~$G(n,p)$ can be verified.
\begin{lemma}\label{lem:EN2G}
	For all integers $s$, $t\geq 2$ and every $d>0$ there exist $\eps>0$ and $C>0$ such that for 
	$p>Cn^{-1/m_2(K_{s})}$ the following holds a.a.s.\
	for $G\in G(n,p)$.

	If $H\subseteq G$ with $e(H)=d|E(G)|$ and with the property that 
	every balanced $t$-cut has size at most $(1-1/t+\eps)d e(G)$, then $H$ contains a copy of~$K_s$.
\end{lemma}
Finally,  a standard application of  Lemma~\ref{lem:EN2G} with $s=k+1$ and $t=k-1$
yields Theorem~\ref{thm:probEN}. We omit the details here.

\section{Concluding Remarks}
We close with a few comments of related results and open problems.
%exclude from Table of Contents}
\tocless{\subsection*{Related Results}}
We restricted the discussion to extremal question in random graphs.
However, the results of Conlon and Gowers~\cite{CG} and Schacht~\cite{Sch} 
and also the subsequent work of Samotij~\cite{Sa}, Balogh, Morris, Samotij~\cite{BMS}, and 
Saxton and Thomason~\cite{ST} applied in a more general context and led 
to extremal results for random hypergraphs and random subsets of the integers.
Here we state a probabilistic version of Szemer\'edi's theorem on arithemtic progressions~\cite{Sz75} (see Theorem~\ref{thm:probSz} below).

For integers $k\geq 3$ and $n\in \NN$,  and a set $A\subseteq \ZZ/n\ZZ$, let $r_k(A)$ 
denote the cardinality of a maximum subset of $A$, which contains no arithmetic progression of length~$k$, i.e., 
\[
	r_k(A)=\{|B|\colond B\subseteq A\tand B\ \text{contains no arithmetic progresion of length $k$}\}
\]

 Answering a well known conjecture of Erd\H os and Tur\'an~\cite{ET36}, 
 Szemer\'edi's theorem asserts that 
 \[
 	r_k(\ZZ/n\ZZ)=o(n)
 \] 
 for every integer $k\geq 3$. The following probabilistic version 
 of Szemer\'edi's theorem was obtained for $k=3$ by Kohayakawa, \L uczak, and R\"odl~\cite{KLR96}
 and for all~$k$ in~\cite{CG,Sch}. 
\begin{theorem}
	\label{thm:probSz}
	For every integer $k\geq 3$ and every $\eps>0$ the function $\ph=n^{-1/(k-1)}$ is 
	a threshold for $\cS_n(k,\eps)=\{A\subseteq \ZZ/n\ZZ\colond r_k(A)\leq \eps |A|\}$.
\end{theorem}
Note that similarly as for the threshold for the Erd\H os--Stone theorem for random graphs, the threshold 
for Szemer\'edi's theorem coincides with that $p$ for which a random subset of $\ZZ/n\ZZ$ has in expectation the same number of elements and number of arithmetic progressions of length~$k$.

Let us remark that the methods from~\cite{CG,Sch} also can be used to derive thresholds
 for Ramsey properties 
for random hypergraphs and random subsets of the integers (see~\cite{CG, FRS10} for details).

%exclude from Table of Contents}
\tocless{\subsection*{Open Problems}}
Besides these recents advances,  several important questions are still unresolved. For example, 
it would be very interesting if the result of DeMarco and Kahn~\cite{DMK}
(Theorem~\ref{thm:probMantel}) could be extended to cliques of arbitrary size 
(see Conjecture~\ref{conj:DMK}). 
%Also a generalization of Theorem~\ref{thm:OPT} would be of high interest (see Question~\ref{quest:OPT}). 
Finally, we would like to point out that 
for some applications (see, e.g., Theorem~\ref{thm:probSiSo}) a generalization of 
part~\ref{it:cnt2Gnp} of Theorem~\ref{thm:cntGnp} for all graphs~$F$ 
would be useful.

%\bibliographystyle{amsplain}
%\bibliography{erdos100}

\providecommand{\bysame}{\leavevmode\hbox to3em{\hrulefill}\thinspace}
\providecommand{\MR}{\relax\ifhmode\unskip\space\fi MR }
\providecommand{\MRhref}[2]{%
  \href{http://www.ams.org/mathscinet-getitem?mr=#1}{#2}
}
\providecommand{\href}[2]{#2}

\end{document}